\documentclass[a4paper]{article}
\usepackage [english]{babel}
\usepackage{amsmath} 
\usepackage{amsfonts} 
\usepackage{amssymb}
\usepackage{amsthm}
\usepackage{amscd}
\usepackage{latexsym}
\usepackage{epsfig}
\usepackage{color}
\usepackage{cancel}
\usepackage{hyperref}
\usepackage{mathrsfs}
\usepackage{epsfig}
\usepackage{graphics}
\usepackage{graphicx}
\usepackage{MnSymbol}
\usepackage{datetime}
\usepackage{chngcntr}
\usepackage{apptools}
\usepackage{enumerate}

\newtheorem{thm}{Theorem}[section]

\newtheorem{remark}[thm]{Remark}

\newtheorem{lem}[thm]{Lemma}
\AtAppendix{\counterwithin{lemma}{section}}
\AtAppendix{\counterwithin{thm}{section}}
\newtheorem{lemma}{Lemma}

\newtheorem{cor}[thm]{Corollary}

\newcommand{\dx}{\, dx}

\newcommand{\D}{\partial}

\newcommand{\R}{\mathbb{R}}

\providecommand{\dv}{\,\mathrm{div}\,}

\renewcommand{\div}{\textrm{div}}

\definecolor{mypink}{RGB}{219, 48, 122}
\definecolor{myblue}{RGB}{0, 0, 122}
\definecolor{myred}{rgb}{0.86, 0.08, 0.24}

  \hypersetup{
           breaklinks=true,   
           colorlinks=true,   
           pdfusetitle=true,  
        }

\date{\empty}

\begin{document}

  \title{Global existence for a 3D Tropical Climate Model with damping
    and small initial data in $\dot H^{\frac{1}{2}}(\R^3)$}
    
\author{
Diego Berti\footnote{Universit\`a di Pisa, Dipartimento di Matematica, Largo
    Bruno Pontecorvo, 5, 56127 Pisa, Italia, \newline Email: diego.berti@dm.unipi.it}
\and
Luca Bisconti\footnote{Universit\`a di Firenze, Dipartimento di Matematica e
    Informatica ``U.~Dini'', Viale Morgagni 67/a, 50134 Firenze, Italia,\newline Email: luca.bisconti@unifi.it}
\and
Davide Catania\footnote{Universit\`a eCampus, Facolt\`a di Ingegneria, Via
    Isimbardi\, 10, I-22060 Novedrate (CO), Italia,\newline Email: davide.catania@unibs.it}
  }  

\maketitle

  \begin{abstract}
We consider a 3D Tropical Climate Model with damping
terms in the equation of the barotropic mode $u$ and in the equation
of the first baroclinic mode $v$ of the velocity.  The equation for
the temperature $\theta$ is free from dampings.  We prove global
existence in time for this system assuming the initial data $(u_0,
v_0,\theta_0)$ small, in terms of the homogeneous space $\dot
H^{\frac{1}{2}}(\R^3)$.
\end{abstract}

\vspace{0.5 cm}


\section{Introduction} \label{sec:intro}
We consider the following 3D tropical Climate Model with damping, i.e.
\begin{equation}\label{TCM-gen}
  \begin{aligned}
    &\D_t u+ (u\cdot \nabla) u -\nu\Delta u + \sigma_1 |u|^{\alpha -1}u +\nabla \pi+\dv  (v \otimes v)=0,\\
    &\D_t v+ (u \cdot \nabla) v -\eta\Delta v + \sigma_2 |v|^{\beta -1}v + (v\cdot \nabla) u  +\nabla\theta =0,\\
    & \D_t \theta + (u \cdot \nabla) \theta  -\mu \Delta \theta  + \dv  v =0,\\
    &\dv  u = 0,\\
    & u (x, 0) =u_0,\,\, v (x, 0)=v_0,\,\, \theta (x, 0)=\theta_0,
  \end{aligned}
\end{equation}
where, for $x\in \R^3$ and $t \geq 0$, $u = (u_1(x,t), u_2(x,t), u_3(x,t))$
and $v = (v_1(x, t), v_2(x, t), v_3(x,t))$ denote, respectively, the barotropic mode
and the first baroclinic mode of the velocity, $\theta = \theta(x,t)$ indicates
the temperature, and $\pi=\pi(x,t)$ represents the pressure that acts as a Lagrange multiplier for this system.
Also, $\nu>0$, $\eta>0$, and  $\mu>0$.

Here, we assume $\sigma_1, \sigma_2 > 0$, while
$\alpha, \beta \geq 1$. In particular, $\alpha$ and $\beta$ are the
damping coefficients, for which further appropriate restrictions will
be introduced later.
When $\sigma_1=\sigma_2=0$, the damping terms in \eqref{TCM-gen} are
dropped and the system takes the form of the viscous Tropical Climate
Model. The original version of this model, without diffusion
(i.e. $\nu = \eta =\mu =0$) and set in $\mathbb{R}^2$, was
derived first by Frierson, Majda and Pauluis \cite{Frierson} (see also
\cite{Maj}), while for the analysis of the viscous version of it
(still in $\mathbb{R}^2$) we refer the reader to Li and
Titi~\cite{Li-Titi} (see also \cite{Li-Titi-1}).

As its name suggests, in \cite{Frierson} the system arises from the
context of meteorology, in formalizing the interaction of large scale
flows and precipitations in the tropical atmosphere. Roughly speaking,
system~\eqref{TCM-gen}, with $\sigma_1=\sigma_2=\nu = \eta =\mu =0$,
comes from the hydrostatic Boussinesq equations by performing a vertical
decomposition of the velocity field, followed by a Galerkin truncation
up to its first baroclinic mode.

While the 2D case (without dampings) was covered by the seminal work  \cite{Li-Titi} (see also \cite{Bis-2021, 
Dong-2020, Dong-Wu-2019, Li-Zhai-Yin, Ma-Wan, Niu, Wan, Ye}), in studying~\eqref{TCM-gen} and its variants (in $\mathbb{R}^3$), a main open problem
is the issue of global well-posedness of strong solutions: See \cite{BBC-JMAA, Yuan-Chen, Yuan-Zhang} for the 3D Tropical Climate
Model with damping. For the 3D case without damping but still modified with respect 
to the standard viscous system (e.g. due to the presence of fractional diffusion or differential regularizing terms)
see the recent papers \cite{BBC-AA, Chen-Yuan-Zhang, Zhang-Xu, Zhu}. In \cite{Ye-Zhu}, the authors deal with 
regularity under small initial data in the case of temperature-dependent systems.

Without damping terms, what is known is the local existence of strong
solutions \cite{Ma-Jiang-Wan} and {the global existence} under
small initial data (see \cite{Wang-Zhang-Pan, Zhang-Xu}). The global well-posedness of
solutions to a 2D tropical climate model with dissipation in the
equation of the first baroclinic mode of the velocity, under the
hypotheses of small initial data, was studied by Wan \cite{Wan}, Ma
and Wan \cite{Ma-Wan}.

In the case of a \emph{large damping effect}, namely if $\alpha, \beta \ge 4$, the related
terms help to deal with a priori and energy estimates, having a direct
regularizing action.  In \cite{Yuan-Zhang}, under the previous
assumptions, the authors prove global existence and
uniqueness. With similar techniques, but only requiring
$\alpha \geq 4$, in \cite{Yuan-Chen} the authors still obtain the same
results. On the other hand, if $1\leq \alpha, \beta < 4$,
the situation is, in principle, rather far from the previous ones, and
in fact the presence of the damping terms seems to play a more subtle
role. About this point, we refer to \cite{BBC-JMAA} where
the authors establish a regularity criterion in Besov spaces,
under the assumption $3\le \alpha, \beta <4$.

  In this paper, requiring $5/2\leq \alpha, \beta < 4$, we give a new
   regularity criterion, in terms of small initial data in
   $\dot H^{\frac{1}{2}}(\R^3)$, actually bounded in
   $H^2(\mathbb{R}^3)$ and sufficiently small in
   $\dot H^{\frac{1}{2}}(\R^3)$, for the global strong solutions to
   system \eqref{TCM-gen} (see Theorem~\ref{main} below). The existence,
   regularity and uniqueness of local solutions is given in Lemma
   \ref{bound-H2}.
    
   In order to prove Theorem \ref{main}, we first obtain a new
   blow-up criterion (see Theorem \ref{th:bu criterion}) in terms of
   proper integrability in time of \textrm{BMO}-norms of the local solution
   $(u,v,\theta)$. To apply Theorem \ref{th:bu criterion}, it
   suffices to bound uniformly in time the $\dot H^{\frac32}$-norm of
   $(u,v,\theta)$ (see Remark~\ref{rmk-Linfty-H3/2}). A key ingredient
   that we use is stated in Corollary \ref{cor:monotone}: By means of
   the smallness hypothesis, we are able to prove that the low-order
   norms in $\dot H^{\frac12}$ and $\dot H^1$ of $(u,v,\theta)$ remain
   bounded by the corresponding norms of the initial data. This is
   obtained combining together the $\dot H^{\frac12}$- and
   $\dot H^1$-energy estimates given in Section~\ref{sec:main}.
    
   The main difficulties arise from the fact that we need to perform
   some delicate estimates in the spaces $\dot H^{\frac12}(\R^3)$ and
   $\dot H^{\frac32}(\R^3)$, and in doing so we 
   consider fractional derivatives of the various terms involved in
   the equations~\eqref{TCM-gen} and, among others, of the damping
   terms (on this regard see also \cite{Chae, Chae-Lee, Zhao,
     Zhang-Xu}). The approach used here derives from the analysis
   proposed in \cite{Chae, Chae-Lee} and, when $\sigma_1=\sigma_2=0$
   (see the Appendix), it still works and allows us to generalize, or
   to provide an alternative proof, to the results given in
   \cite{Zhang-Xu} and \cite{Wang-Zhang-Pan}, respectively.
   \smallskip
    
   The plan of the paper is as follows: In Section
   \ref{sec:preliminaries-main}, we give the basic notations, we
   collect the main tools from literature and we state the main
   outcomes of the paper. Section \ref{sec-stime} takes care of
   providing local well-posedness and the proof of Theorem \ref{th:bu
     criterion}. Section \ref{sec:main} is dedicated to the proof of
   Theorem \ref{main}, passing through the demonstration of Corollary
   \ref{cor:monotone}.

  \section{Preliminaries and main results} \label{sec:preliminaries-main}
  For $p \geq 1$, we indicate by $L^p=L^p(\R^3)$ the usual Lebesgue
  space (not distinguishing scalar from vector fields, with a safe abuse of 
  notations), and with $\|\, \cdot\, \|_p=\|\, \cdot\, \|_{L^p}$ its usual
  norm; we set $\|\, \cdot\, \|= \|\, \cdot\, \|_2 $ when $p=2$.  For
  $s >0$ and $p\ge 1$, we denote by $W^{s,p}=W^{s,p}(\mathbb{R}^3)$
  the Sobolev space, with its dual $W^{-s, p\rq{}}$, where
  $1/p+1/p\rq{}=1$. When $p=2$, we write $H^s=W^{s,2}(\mathbb{R}^3)$
  with norm $\|\cdot\|_{H^s}$; $\dot H^s$ denotes the standard
  homogeneous Sobolev space with norm $\|\, \cdot\, \|_{\dot
    H^s}$ (see, e.g. \cite{Ad}). With $\Lambda^s$ we mean the operator defined (formally) as
  $\Lambda^s f=(-\Delta)^\frac{s}{2} f$, such that
  $\widehat{\Lambda^s f}(\xi)=|\xi|^s\hat{f}(\xi)$, for
  $\xi \in \mathbb{R}^3$. Also, with \textrm{BMO} we denote the space of
  \emph{bounded mean oscillation functions}.
  Moreover, we set $H^2(\R^3) ^3\doteq H^2(\R^3) \times H^2(\R^3)
 \times H^2(\R^3)$. 

 As a further matter of notation, let us introduce 
    $\| (u, v, \theta)\|^2\doteq \| u\|^2 +\| v\|^2+\| \theta\|^2$,
    and the same for  $\|\nabla (u, v, \theta)\|^2$, $\|\Delta (u,
v, \theta)\|^2$, $\|\nabla\Delta (u, v, \theta)\|^2$ and $\|\Lambda^s ( u, v, \theta)\|^2$, $s>0$.
  
 For the remainder of the paper, the symbol $C$ indicates a generic constant
  independent of the solution.
  
  \subsection{Some inequalities} \label{subsec:utility-estimates}
  Here below, we collect the main estimates that we will use in the sequel.
  \begin{itemize}
  \item Gagliardo--Nirenberg's inequality in $\mathbb{R}^3$ (see
    \cite{Nirenberg-1959}, \cite{Zhao}):
\newline
    Let $0 \leq m$, $s \leq l$, then we have
  \begin{equation*}
    \label{e:GN-Hajaiej}
    \|\Lambda^s f\|_{p}\le C \|\Lambda^mf\|_{q}^{1-\kappa}\|\Lambda^{l}f\|_{r}^\kappa,
  \end{equation*}
  where $\kappa \in [0, 1]$ and $\kappa$ satisfies $
    \frac{s}{3}-\frac{1}{p}=\left(\frac{m}{3} -\frac{1}{q}\right)(1-\kappa)
    +\left(\frac{l}{3}-\frac{1}{r}\right) \kappa$.  When $p = \infty$, we require that $0 < \kappa < 1$.

\item
  Kato--Ponce product estimate
  \cite{Kato-Ponce-1988} {(see also \cite{G-2019, KPV1991})}, i.e.:
  \begin{equation}\label{KP}
    \|\Lambda^s(fg)\|_p \leq C\big(\|f\|_{p_1} \|\Lambda^s g\|_{q_1}
    + \|g\|_{p_2}\|\Lambda^s f\|_{q_2} \big),
  \end{equation}
  with $s>0$, $1<p<\infty$, $1< q_1, q_2< \infty$ and
  $1< p_1, p_2\leq \infty$ such that
  $\frac{1}{p} = \frac{1}{p_1} + \frac{1}{q_1} = \frac{1}{p_2} +
  \frac{1}{q_2}$.

\item
   Commutator estimate (see, e.g.,
  \cite{Kato-Ponce-1988, KPV1991}), that is:
  \begin{equation} \label{KPV}
    \begin{aligned}
      \|[\Lambda^s,f]g\|_p &= \|\Lambda^s(fg) - f\Lambda^s g\|_p \\
      &\leq C\big(\|\nabla f\|_{p_1} \|\Lambda^{s-1}g\|_{q_1} +
      \|g\|_{p_2}\|\Lambda^s f\|_{q_2} \big),
    \end{aligned}
  \end{equation}
  with $s>0$, $1<p<\infty$, $1< q_1, q_2< \infty$ and
  $1< p_1, p_2\leq \infty$ such that
  $\frac{1}{p} = \frac{1}{p_1} + \frac{1}{q_1} = \frac{1}{p_2} +
  \frac{1}{q_2}$.
  
  \item   From  \cite{Kozono2000} (see also
  \cite{Zhang-Xu}): 
  \begin{equation} \label{Kozono-BMO} \| \D^\alpha f\cdot \D^\beta
    g\|_r \leq C \big(\|f\|_{{}_{\textrm{BMO}}}\|
    (-\Delta)^{\frac{|\alpha|+|\beta|}{2}}g\|_r + \|g\|_{{}_{\textrm{BMO}}} \|
    (-\Delta)^{\frac{|\alpha|+|\beta|}{2}}f\|_r\big),
  \end{equation}
  where $1 < r < \infty$, $\alpha = (\alpha_1, \alpha_2,\alpha_3)$ and
  $\beta = (\beta_1, \beta_2, \beta_3)$ are multi-indices with
  $|\alpha|, |\beta| \geq 1$,
  $\D^\alpha ={\partial_{x_1}^{\alpha_1}\, \partial_{x_2}^{\alpha_2}\,
    \partial_{x_3}^{\alpha_3}}$ and
  $\D^\beta ={\partial_{x_1}^{\beta_1}\, \partial_{x_2}^{\beta_2}\,
    \partial_{x_3}^{\beta_3}}$.
    
\item From \cite{Kozono2000}, with $1<p<\infty$, we have:
      \begin{equation} \label{BMO} \|f \,\cdot\, g\|_p \leq
    C(\|f\|_p\|g\|_{_{\textrm{BMO}}} + \|f\|_{_{\textrm{BMO}}}\|g\|_p).
  \end{equation}
  
  \item  From \cite[(1.11)]{Kozono2002}, with $s>3/2$:
  \begin{equation}
    \label{e:kozono2002}
    \|f\|_{\infty}\le C\big(1+\|f\|_{_{\textrm{BMO}}}\, \ln (e+ \|f\|_{H^s})\big).
  \end{equation}
  \end{itemize}

  \subsection{The main results} \label{subsec:reg-result}
  
  In Section \ref{sec-stime}, we derive a \emph{blow-up criterion} for the local solution of \eqref{TCM-gen}. 
  
    \begin{thm}[Blow-up criterion] \label{th:bu criterion}  Assume 
      $5/2\leq \alpha, \beta < 4$, and
      $(u_0, v_0, \theta_0) \in H^2(\R^3)^3$, with $\dv u_0 = 0$. Let $(u,v,\theta)(t)$ be the local solution of \eqref{TCM-gen}. Define
  \begin{equation}
  \label{e:Pi}
  \Pi(T) \doteq 
  \int_0^T C \big(1
      +\|u(t)\|_{{}_{\textrm{BMO}}}^{8} + \|v(t)\|_{{}_{\textrm{BMO}}}^{8} +
      \|\theta(t)\|_{{}_{\textrm{BMO}}}^2 +
      \|u(t)\|^{\alpha+1}_{\alpha+1}+\|v(t)\|^{\beta+1}_{\beta+1}
      \big)dt.
  \end{equation}
      
      Then , for $0< T^\ast < +\infty$, we have that  
  \begin{equation*}
  \|\Delta (u, v, \theta) (T)\| < +\infty, \ \mbox{ for every } \ 0<T< T^\ast \ \mbox{ and } \ \limsup_{T \to {T^\ast}} \|\Delta (u, v, \theta) (T)\|=+\infty,
  \end{equation*}
  if and only if 
  \begin{equation*}
  \Pi(T) < +\infty \ \mbox{ for every } \ 0<T<T^\ast \ \mbox{ and } \ \Pi(T^\ast) =+\infty.
  \end{equation*}
  \end{thm}
  
    If a time $T^\ast>0$ as in Theorem~\ref{th:bu criterion}
  does exist, it is called the \emph{first blow-up time}. Observe
  that, for every $0<T<T^\ast$, we also have
  \begin{equation*}
    \int_{0}^{T}\|\nabla \Delta (u,v,\theta)(t)\|^2\,dt < +\infty.
  \end{equation*}
  
  In Section \ref{sec:main}, by means of energy estimates, we make use
  of Theorem~\ref{th:bu criterion} in order to obtain the following
  global existence result,
  in terms of $\dot H^{\frac12}$-smallness of the initial data.
  \begin{thm}[Global existence and regularity under small initial data] \label{main} Assume
      $5/2\leq \alpha, \beta < 4$, and
      $(u_0, v_0, \theta_0) \in H^2(\R^3)^3$, with $\dv u_0 = 0$.  Let
      $R>0$ be such that $\|u_0\|_{ H^2}+\|v_0\|_{H^2}+ \|\theta_0\|_{ H^2} \le R$.
      Then, there exists a sufficiently small constant $c_0 =c_0(R) > 0$ such that, if
    \begin{equation} \label{constant-c0}
      \|u_0\|_{\dot H^{\frac{1}{2}}} + \|v_0\|_{\dot H^{\frac{1}{2}}} +
      \|\theta_0\|_{\dot H^{\frac{1}{2}}}\leq c_0 ,
    \end{equation}
    the system~\eqref{TCM-gen} admits a unique global-in-time strong
    solution $(u, v, \theta)$, satisfying
    \begin{equation*}
      u,v,\theta \in L^\infty([0, T); H^2(\R^3))\cap
      L^2([0, T); H^3(\R^3)) \ \mbox{ for every } \ T>0.
    \end{equation*}
  \end{thm}
  
  Note that the necessity to bound
  $\|u_0\|_{H^2}+\|v_0\|_{H^2}+\|\theta_0\|_{H^2}\le R$ is an
  inheritance of Corollary \ref{cor:monotone}. Hence, as in Corollary
  \ref{cor:monotone}, we actually do not need this bound in the case
 {$5/2 \le \alpha, \beta \le 3$}, when the conclusion of Theorem
  \ref{main} follows with $c_0$ not depending on $R$. In the other
  cases, we cannot remove the assumption, {since
    $c_0(R) =o(1)$ as $R \to +\infty$.}
  
  \section{Local well-posedness and blow-up criterion}\label{sec-stime}
We begin with the local existence and uniqueness of strong solutions.

  \begin{lem}[Local well-posedness] \label{bound-H2}
    Under the assumptions of Theorem \ref{main}, there exists a
    positive time $\hat T>0$ and a unique strong solution
    $(u, v, \theta)(t)$ on $[0, \hat T)$ to the system~\eqref{TCM-gen},
    satisfying
    \begin{equation*}
      (u, v, \theta) \in L^\infty([0, T]; H^2(\R^3))\cap L^2([0, T]; H^3(\R^3))
      \,\, \textrm{ for every }\,\, 0< T <\hat T.
    \end{equation*}
  \end{lem}

  \begin{proof}
  We proceed formally, in order to obtain a priori estimates. A
  standard approximation procedure gives then the desired result (see
  e.g. \cite{Chae}, see also \cite{Ma-Jiang-Wan}).
  
  Let us start by taking the $L^2$-inner product of
  $\eqref{TCM-gen}_1$--$\eqref{TCM-gen}_3$,
  respectively, with $u$, $v$ and $\theta$, and then
  by adding the resulting equations up. This gives
  \begin{equation*} 
    \frac{d}{dt}\|(u,v,\theta)(t)\|^2
    +\nu\,\|\nabla u\|^2+\eta\, \|\nabla v\|^2 +\mu\,\|\nabla
    \theta\|^2 + \sigma_1\, \|u\|_{\alpha+1}^{\alpha+1}+\sigma_2 \,
    \|v\|_{\beta+1}^{\beta+1} =0,
  \end{equation*}
  from which we get, after integrating in time variable in the interval
  $(0,T)$,
  \begin{equation} \label{bound-L2-integral}
    \|(u,v,\theta)(T)\|^2 + \int_{0}^{T} \left( \|\nabla (u,v,\theta)(t)\|^2
      + \|u(t)\|_{\alpha+1}^{\alpha+1}+ \|v(t)\|_{\beta+1}^{\beta+1}
    \right) \, dt \le C\, \|(u_0,v_0,\theta_0)\|^2,
  \end{equation}
  where we recall that
$\|\nabla (u, v, \theta)(t)\|^2 =\|\nabla u(t)\|^2 +\|\nabla
v(t)\|^2+\|\nabla \theta(t)\|^2$.
As a consequence of the local existence, and in particular thanks to 
\eqref{bound-L2-integral}, we will also have that $u\in L^{\alpha+1}([0,T); L^{\alpha+1}(\mathbb{R}^3))$ and
$ v\in L^{\beta+1}([0,T); L^{\beta+1}(\mathbb{R}^3))$, for every $0 < T<\hat T$. 
  
  Further, applying $\Delta$ to $\eqref{TCM-gen}_1$--$\eqref{TCM-gen}_3$ and multiplying
  the resulting equations in $L^2$ by $\Delta u$, $\Delta v$ and
  $\Delta \theta$, respectively, lead to
  \begin{equation} \label{Hs-stima-diff-iniziale}
    \begin{aligned}
      \frac{1}{2} \frac{d}{dt} \|\Delta & (u, v, \theta) (t) \|^2
      + \nu\|\nabla \Delta u\|^2 + \eta \|\nabla \Delta v\|^2 + \mu
      \|\nabla \Delta \theta\|^2
      \\
      = & -\int_{\R^3}\Delta \big( (u\cdot \nabla) u\big) \cdot \Delta
      u \dx - \int_{\R^3}\Delta \big( (u\cdot \nabla) v\big) \cdot
      \Delta v \dx - \int_{\R^3}\Delta\big( (u \cdot \nabla)
      \theta\big) \cdot \Delta\theta \dx
      \\
      &-\int_{\R^3}\Delta\big( (v\cdot \nabla) v\big) \cdot \Delta u
      \dx - \int_{\R^3}\Delta\big( v \,\dv v\big)\cdot \Delta u \dx -
      \int_{\R^3}\Delta\big( (v\cdot \nabla) u\big) \cdot \Delta v \dx
      \\
      & - \sigma_1\int_{\R^3}\Delta(|u|^{\alpha-1}u)\cdot \Delta u
      \dx- \sigma_2\int_{\R^3}\Delta(|v|^{\beta-1}v)\cdot \Delta v \dx
      \doteq \sum_{i=1}^8 I_i,
    \end{aligned}
  \end{equation}
where we used the identity
  \begin{equation*} 
    \int_{\R^3}\Delta \nabla \theta\cdot \Delta v dx
    + \int_{\R^3}\Delta(\dv v)\cdot \Delta \theta dx =0.
  \end{equation*}
  Also, using the fact that $\dv u =0$ (see $\eqref{TCM-gen}_4$), we have, by means of the (continuous) embedding
  $H^1(\mathbb{R}^3)\hookrightarrow L^6(\mathbb{R}^3)$ and
  $\|\nabla w\|\le C\|w\|^{\frac14}\,\|\Delta w\|^{\frac34}$ as in
  \cite{Ma-Jiang-Wan},
  \begin{equation*}
  \label{e:I-no-dampings}
    \begin{aligned}
      |I_i| &= \Big|\int_{\mathbb{R}^3} \Delta \big( (u\cdot \nabla)
      w\big) \cdot \Delta w \dx \Big| \leq C \|\nabla \Delta
      u\|\,\|\Delta w\|^2, \,\,\,\, &i=1,2,3,
      \\
      |I_i|&\leq C\big(\|\nabla \Delta u\| + \|\nabla \Delta
      v\|\big)\,\|\Delta (u,v,\theta)\|^2, \,\,\,\, &i=4,5,6,
    \end{aligned}
  \end{equation*}
   where   $w$ plays the role of $u$, $v$ or $\theta$.
    
   The contributions of the terms involving the damping terms are instead dealt with as follows.
   First, we consider $I_7$. From Holder's and Young's inequalities, the embedding
   $H^1(\mathbb{R}^3)\hookrightarrow L^6(\mathbb{R}^3)$ and Gagliardo--Nirenberg\rq{}s inequality
  \begin{equation} \label{e:loc-est-alpha}
    \|u\|_{3(\alpha-1)}\le C\, \|u\|^{1-\kappa}\, \|\Delta u\|^{\kappa}\le C\,  \|\Delta u\|^{\kappa} \
    \mbox{ with } \ \kappa(\alpha)=\frac{3\alpha-5}{4(\alpha-1)} \in (0,1),
  \end{equation}
  we get
    \begin{align*}
      \Big| \int_{\R^3}\Delta(|u|^{\alpha-1}u)\cdot \Delta u \dx\Big|
      & \le \|\nabla \big(|u|^{\alpha-1}u\big)\|\, \|\nabla \Delta u\|
      \\
      &\le C\,\||u|^{\alpha-1}\nabla u\|\, \|\nabla \Delta u\|
      \\
      &\le C\, \||u|^{\alpha-1}\|_3^2\, \|\nabla u\|_6^2 + \varepsilon
      \|\nabla \Delta u\|^2
      \\
      &\le C\, \|u\|_{3(\alpha-1)}^{2(\alpha-1)}\, \|\nabla u\|_6^2 +
      \varepsilon \|\nabla \Delta u\|^2
      \\
      &\le C\, \|\Delta u\|^\frac{3\alpha-5}{4(\alpha-1)}\, \|\Delta
      u\|^2+ \varepsilon \|\nabla \Delta u\|^2
      \\
      &\le C(1+\|\Delta u\|^3) + \varepsilon \|\nabla \Delta u\|^2.
          \end{align*}
          Similarly, for $I_8$ we obtain
          \begin{equation*}
      \Big| \int_{\R^3}\Delta(|v|^{\beta-1}v)\cdot \Delta v\dx\Big|  
     \le C(1+\|\Delta v\|^3) + \varepsilon \|\nabla \Delta v\|^2.
     \end{equation*}

  Plugging the above estimates in \eqref{Hs-stima-diff-iniziale}, and
  set $\ell \doteq \min \{\nu, \eta, \mu\}$, we obtain
  \begin{equation} \label{e:local0}
    \frac{1}{2} \frac{d}{dt} \|\Delta  (u, v, \theta) (t) \|^2
    + \ell\|\nabla \Delta(u, v, \theta)(t)\|^2
    \leq C\,
    \|\nabla \Delta (u,v)\|\,\big(1 + \|\Delta (u, v, \theta) (t)
    \|^3 \big).
  \end{equation}
 Then, an application of Young's inequality to the right-hand
  side of \eqref{e:local0}, and straightforward manipulations lead to
  \begin{equation*} \label{e:local}
    \frac{1}{2} \frac{d}{dt}\big(1+ \|\Delta  (u, v, \theta) (t) \|^2\big)
    + \big(\ell-\varepsilon\big)\|\nabla \Delta(u, v, \theta)(t)\|^2
    \leq C\,\big(1 + \|\Delta (u, v, \theta) (t)
    \|^2 \big)^3.
  \end{equation*}
    
  Now, set $X(t) = 1 + \|\Delta (u, v, \theta) (t) \|^2$ and
  \begin{equation*}
    \hat T = \frac{3}{8\,C_0\,X^2(0)}.
  \end{equation*}
  As a consequence of
   
    \begin{equation*}
      \frac{d}{dt} X\leq C_0 X^3\,\,\, \textrm{ and so } \,\,\,
      X(t)\leq \frac{X(0)}{\sqrt{1-2\,C_0\,X^2(0)\, t}},
    \end{equation*}
    we obtain that $X(t)\leq 2 X(0)$, for $t \in [0,T]$, and for
    every $0<T<\hat T$.

 {Uniqueness and continuous dependence are even more standard and can be achieved, for instance, as done
  in \cite{Yuan-Chen} (see also \cite{Yuan-Zhang})}.
\end{proof}

\begin{remark}
    The same conclusion of Lemma~\ref{bound-H2} holds true if $\alpha,
\beta > 5/3$, since \eqref{e:loc-est-alpha} (and the corresponding
estimate involving $\beta$) still applies.
  \end{remark}

  \subsection{Proof of Theorem \ref{th:bu criterion}}
  In this section, we derive Theorem \ref{th:bu criterion}. In
order to do so, let us now turn back to the differential relation
\eqref{Hs-stima-diff-iniziale}, and recast our estimates in order to
make explicit the just mentioned blow-up criterion.
  
Recalling the identities
  \begin{equation*}
    \int_{\R^3}\Delta \big( (u\cdot \nabla) u\big) \cdot
    \Delta u \dx =
    \int_{\R^3} \nabla^2 \big((u\cdot\nabla) u\big) \cdot \nabla^2 u\dx
  \end{equation*}
  and
  \begin{equation*}
    \int_{\R^3} (u\cdot\nabla)\nabla^2 u \cdot \nabla^2 u\dx =0,
  \end{equation*}
  using H\"older's inequality and exploiting \eqref{Kozono-BMO}, we
  reach
  \begin{equation} \label{I1}
    \begin{aligned}
      |I_1| &= \left| \int_{\R^3} \big[\nabla^2 (u\cdot\nabla)u \cdot\nabla^2 u -
        (u\cdot\nabla)\nabla^2 u \cdot \nabla^2 u\big]\dx \right|
      \\
      &\leq \hspace{-0.8 cm}
      \sum_{{\tiny\left. \begin{array}{l} \qquad|\lambda|\geq 1\\
                           \qquad |\lambda|+|\gamma|=2
                         \end{array} \right.}}\!\!\!
                   C\|(\nabla^\lambda u\cdot \nabla ) \nabla^\beta u\| \|\nabla^2 u\|
      \leq  \hspace{-0.8 cm}
      \sum_{{\tiny\left. \begin{array}{l} \qquad|\lambda|\geq 1, \,\, |\bar \gamma| \geq 1\\
                           \qquad |\lambda|+|\bar\gamma|=3
                         \end{array}\right.}}\!\!\!
                   C\|\D^\lambda u\, \D^{\bar \gamma} u\|\|\Delta u\|
  \\
  &\leq C \|u\|_{{}_{\textrm{BMO}}}\|\nabla\Delta u\|\|\Delta u\|
  \\
  &\leq C \|u\|_{{}_{\textrm{BMO}}}^2 \|\Delta u\|^2 +\varepsilon
  \|\nabla\Delta u\|^2,
\end{aligned}
\end{equation}
where 
$\displaystyle  \nabla^2 \doteq \D_{x_1}^{\lambda_1}\D_{x_2}^{\lambda_2}\D_{ x_3}^{\lambda_3}$,  with 
$\lambda_1+\lambda_2+ \lambda_3=2$.
      
Similarly, we get
\begin{equation*}
  \begin{aligned}
    |I_2| &\leq \left| \int_{\R^3}\left[\Delta \big( (u \cdot \nabla)
        v\big) - (u \cdot \nabla)\Delta v \right] \cdot \Delta v \dx
    \right|
    \\
    & \leq C\big( \|u\|_{{}_{\textrm{BMO}}}\|\nabla\Delta v\|
    +\|v\|_{{}_{\textrm{BMO}}}\|\nabla \Delta u\|\big)\|\Delta v\|
    \\
    &\leq C\big(1 + \|u\|_{{}_{\textrm{BMO}}}^2+ \|v\|_{{}_{\textrm{BMO}}}^2 \big)
    \|\Delta v\|^2 + \varepsilon \big(\|\nabla\Delta
    u\|^2+\|\nabla\Delta v\|^2 \big),
  \end{aligned}
\end{equation*}
and also
\begin{equation*}
  \begin{aligned}
    |I_3| &\leq \left| \int_{\R^3}\left[\Delta \big( (u \cdot \nabla)
        \theta\big) - (u \cdot \nabla)\Delta \theta \right] \cdot
      \Delta \theta \dx \right|
    \\
    & \leq C\big( \|u\|_{{}_{\textrm{BMO}}}\|\nabla\Delta \theta\|
    +\|\theta\|_{{}_{\textrm{BMO}}}\|\nabla \Delta u\|\big)\|\Delta \theta\|
    \\
    &\leq C\big(1 + \|u\|_{{}_{\textrm{BMO}}}^2+ \|\theta\|_{{}_{\textrm{BMO}}}^2 \big)
    \|\Delta \theta\|^2 + \varepsilon \big(\|\nabla\Delta
    u\|^2+\|\nabla\Delta \theta\|^2 \big).
  \end{aligned}
\end{equation*}

  After recalling that
\begin{equation*}
  \int_{\R^3} \big[(v\cdot \nabla ) \nabla^2 v \cdot \nabla^2 u
  +  (v\cdot \nabla ) \nabla^2 u \cdot \nabla^2 v \big]\dx =
  -  \int_{\R^3}  \dv v\, \nabla^2 v \cdot \nabla^2 u \dx,
\end{equation*}
we can reason similarly as for $I_1$ in \eqref{I1}, thus obtaining
\begin{equation} \label{e:I4-I6}
  \begin{aligned}
    |I_4+I_6|& = \left|\int_{\mathbb{R}^3}\nabla^2\big(\big(v\cdot
      \nabla\big)v\big)\cdot \nabla^2 u + \int_{\mathbb{R}^3} \nabla^2
      \big(\big(v\cdot \nabla\big)u\big)\cdot \nabla^2 v \right|
    \\
    &\le \left|\int_{\R^3}\big[ \nabla^2 \big((v\cdot \nabla
      )v\big)\cdot \nabla^2 u - \big((v\cdot \nabla ) \nabla^2
      v\big)\cdot \nabla^2 u \big]\dx\right|
    \\
    & \quad+ \left|\int_{\R^3}\big[ \nabla^2 \big((v\cdot \nabla
      )u\big)\cdot \nabla^2 v - \big((v\cdot \nabla ) \nabla^2
      u\big)\cdot \nabla^2 v \big]\dx\right|
    \\
    & \quad+  \left|\int_{\R^3}  \dv v \,\nabla^2 v \cdot \nabla^2 u \dx\right|\\
    &\leq C \|v\|_{{}_{\textrm{BMO}}}\|\nabla\Delta v\|\|\Delta u\| +C
    \big(\|u\|_{{}_{\textrm{BMO}}}\|\nabla\Delta v\|
    + \|v\|_{{}_{\textrm{BMO}}} \| \nabla \Delta u\|\big)\|\Delta v\|\\
    & \quad +\|\dv v\, \nabla^2 v\|\, \| \nabla^2 u\|\\
    &\leq C\big(1 + \|u\|_{{}_{\textrm{BMO}}}^2+ \|v\|_{{}_{\textrm{BMO}}}^2
    \big)\big(\|\Delta u\|^2+\|\Delta v\|^2 \big) + \varepsilon
    \big(\|\nabla\Delta u\|^2+\|\nabla\Delta v\|^2 \big),
  \end{aligned}
\end{equation}
where for the last term we used again \eqref{Kozono-BMO}.

Considering $I_{5}$, we have
\begin{equation}
  \begin{aligned}
    \label{I5}
    |I_{5}|&\leq \|\nabla (v\dv v )\|\|\nabla \Delta u\|
    \\
    & \leq \big(\|\dv v \nabla v \| +\| v \nabla \dv v \| \big)
    \|\nabla \Delta u\|.
  \end{aligned}
\end{equation}

Thanks to \eqref{e:kozono2002} and Gagliardo--Nirenberg's inequality
$\|\Delta v\| \leq C \|v\|^{\frac13} \|\nabla\Delta v\|^{\frac23} \leq
C\|\nabla\Delta v\|^{\frac23}$, we get
\begin{equation*}
  \begin{aligned}
    \|v \, \nabla\dv v\| \|\nabla \Delta u\| &\le
    C\|v\|_{\infty}\|\Delta v\| \|\nabla \Delta u\|
    \\
    &\le C\Big( 1+\|v\|_{_{\textrm{BMO}}}\, \ln\big(e+\| \Delta v\|\big)\Big)
    \|\Delta v\|\, \|\nabla \Delta u\|
    \\
    &\le C \Big(1+ \|v\|_{_{\textrm{BMO}}}\,\big(1+\|\Delta
    v\|^{\frac13}\big)\Big)\|\Delta v\|\, \|\nabla \Delta u\|
    \\
    &\le C\big(1+\|v\|_{_{\textrm{BMO}}}^2\big)\big(1+\|\Delta v\|^{\frac23}\big)
    \| \Delta v\|^2+ \varepsilon\, \|\nabla \Delta u\|^2
    \\
    &\le C\big(1+\|v\|_{_{\textrm{BMO}}}^2\big)\big(1+\|\Delta v\|^{\frac23}\big)
    \| \nabla \Delta v\|^{\frac43}
    + \varepsilon \|\nabla \Delta u\|^2 \\
    & \le C\big(1+\|v\|_{_{\textrm{BMO}}}^{6}\big)\big(1+\|\Delta v\|^{2}\big)
    + \, \varepsilon (\|\nabla \Delta u\|^2+\|\nabla \Delta v\|^2).
  \end{aligned}
\end{equation*}

Regarding the former addendum in the right-hand side of \eqref{I5},
by virtue of \eqref{Kozono-BMO}, we get
\begin{equation*}
  \begin{aligned}
    {\|\dv v \nabla v\| \, \|\nabla \Delta u \|}&\le C\,
    \|v\|_{_{\textrm{BMO}}}\,\|\Delta v\|\, \|\nabla \Delta u\|
    \\
    &\le C\, \|v\|^2_{_{\textrm{BMO}}}\, \|\Delta v\|^2 + \varepsilon \|\nabla \Delta     u\|^2.
  \end{aligned}
\end{equation*}
Therefore, we get
\begin{equation}
\label{e:I5}
  |I_{5}|  \le C\big(1 + \|v\|_{_{\textrm{BMO}}}^6 \big)
  \big(1+ \|\Delta v\|^2 \big)  \, + \, \varepsilon (\|\nabla \Delta u\|^2+\|\nabla \Delta v\|^2).
\end{equation}

We now estimate $I_7$, since calculations for $I_8$ are analogous. We
have
\begin{equation} \label{I7}
  \begin{aligned}
    |I_7|& \leq \left|\int_{\R^3}\nabla (|u|^{\alpha-1}u)\cdot \nabla
      \Delta u dx\right|
    \\
    & \leq C\|\nabla (|u|^{\alpha-1}u)\|^2 + \varepsilon\| \nabla
    \Delta u \|^2
    \\
    & \leq C \| |u|^{\alpha-1}\nabla u\|^2 + \varepsilon\| \nabla
    \Delta u \|^2
    \\
    &\leq C \| |u|^{\alpha-1}\|_p^2\|\nabla u\|_{\frac{2p}{p-2}}^2 +
    \varepsilon\| \nabla \Delta u \|^2
    \\
    & = C \| |u|^2\|_{\frac{p(\alpha-1)}{2}}^{\alpha-1} \|\nabla
    u\|_{\frac{2p}{p-2}}^2 + \varepsilon\| \nabla \Delta u \|^2
    \\
    &\leq C\|u\|_{{}_{\textrm{BMO}}}^{\alpha
      -1}\|u\|_{\frac{p(\alpha-1)}{2}}^{\alpha-1}\|\nabla
    u\|_{\frac{2p}{p-2}}^2 + \varepsilon\| \nabla \Delta u \|^2 \doteq
    J_{71} + \varepsilon\| \nabla \Delta u \|^2,
  \end{aligned}
\end{equation}
where, in the last step, we used \eqref{BMO}.  To conclude, we provide
suitable estimates for the term $I_{71}$ in \eqref{I7}.

Then selecting
\begin{equation*}
  p=\frac{2(\alpha+1)}{\alpha-1}> 2\,\,\quad  \textrm{ and }\,\,\quad \frac{2p}{p-2}=\alpha+1 >2, 
\end{equation*}
we have
\begin{equation} \label{J71}
  \begin{aligned}
    |I_{71}|& = \|u\|_{{}_{\textrm{BMO}}}^{\alpha
      -1}\|u\|_{\alpha+1}^{\alpha-1}\|\nabla u\|_{\alpha+1}^2
    \\
    &\leq C
    \|u\|_{{}_{\textrm{BMO}}}^{\alpha-1}\|u\|_{\alpha+1}^{\alpha-1}\big(\|u\|^{1-\delta}
    \|\Delta u\|^\delta \big)^2
    \\
    &\leq C\big( \|u\|_{{}_{\textrm{BMO}}}^{\frac{\alpha^2-1}{2}}
    +\|u\|_{\alpha+1}^{\alpha+1}\|\big) \|\Delta u\|^{2\delta}
    \\
    &\leq C\big(1+\|u\|_{{}_{\textrm{BMO}}}^{8} +
    \|u\|_{\alpha+1}^{\alpha+1}\Big)(1+\|\Delta u\|^2),
  \end{aligned}
\end{equation}
where $0<\delta<1$ is given by Gagliardo--Nirenberg\rq{}s inequality and its
value is $\frac{5\alpha-1}{4(\alpha+1)}$, and
$\frac{\alpha^2-1}{2} < 15/2 <8$.
    
Finally, using \eqref{J71}, we have that
\begin{equation*}
  \begin{aligned}
    |I_7| \leq C \big( 1 +\|u\|_{{}_{\textrm{BMO}}}^{8} +
    \|u\|^{\alpha+1}_{\alpha+1}\big) \big(1 +\|\Delta u\|^2 ) +
    \varepsilon\| \nabla \Delta u \|^2.
  \end{aligned}
\end{equation*}
    
Similarly, we deduce that
\begin{equation} \label{I8} |I_8| \leq C \big( 1 +\|v\|_{{}_{\textrm{BMO}}}^{8}
  + \|v\|^{\beta+1}_{\beta+1}\big) \big(1 +\|\Delta v\|^2 ) +
  \varepsilon\| \nabla \Delta v \|^2.
\end{equation}

Therefore, plugging \eqref{I1}--to--\eqref{I8} in \eqref{Hs-stima-diff-iniziale}, we get
\begin{equation} \label{e:I-dampings}
  \begin{aligned}
    \frac{d}{dt} \big(& 1 + \|\Delta (u, v, \theta) (t) \|^2\big) +
    \big(\min(\nu, \eta, \mu) - 8\varepsilon\big)\|\nabla \Delta( u,
    v, \theta)\|^2
    \\
    &\leq C \Big( 1 +\|u\|_{{}_{\textrm{BMO}}}^{8} +\|v\|_{{}_{\textrm{BMO}}}^{8} +
    \|\theta\|_{{}_{\textrm{BMO}}}^2 +
    \|u\|^{\alpha+1}_{\alpha+1}+\|v\|^{\beta+1}_{\beta+1} \Big)
    \big(1 + \|\Delta (u, v, \theta) (t) \|^2\big),
  \end{aligned}
\end{equation}
and we recall that
$\|\nabla \Delta( u, v, \theta)\|^2 = \|\nabla \Delta u\|^2 + \|\nabla
\Delta v\|^2 + \|\nabla \Delta
\theta\|^2$.

Hence, by using \eqref{e:I-dampings} and Gronwall's inequality, we reach
\begin{equation} \label{capital-pi}
  \begin{aligned}
    \sup_{0\leq t\leq T} \big( 1 + \|\Delta (u, v, \theta) (t)
    \|^2\big) &\leq \big( 1 + \|\Delta (u_0, v_0, \theta_0) \|^2\big)
    \\
    & \hspace{0.5 cm}\times \exp\underbrace{\int_0^T C \big(1
      +\|u\|_{{}_{\textrm{BMO}}}^{8} + \|v\|_{{}_{\textrm{BMO}}}^{8} +
      \|\theta\|_{{}_{\textrm{BMO}}}^2 +
      \|u\|^{\alpha+1}_{\alpha+1}+\|v\|^{\beta+1}_{\beta+1}
      \big)dt}_{= \Pi(T)},
  \end{aligned}
\end{equation}
where $\Pi(T)$ is defined in \eqref{e:Pi}.

As a consequence of \eqref{capital-pi} and of the fact
that, from \eqref{bound-L2-integral}, 
$u\in L^{\alpha+1}([0,T); L^{\alpha+1}(\mathbb{R}^3))$ and also
$ v\in L^{\beta+1}([0,T); L^{\beta+1}(\mathbb{R}^3))$, for any $T>0$, 
we  proved Theorem \ref{th:bu criterion}.

\section{Proof of Theorem \ref{main}}
\label{sec:main}

We begin with the next remark.

\begin{remark} \label{rmk-Linfty-H3/2}
  In order to prove Theorem~\ref{main}, we make use of Theorem
  \ref{th:bu criterion} by means of the control of
  \begin{equation} \label{e:holds} \sup_{0\leq t<
      T^\ast}\|(u,v,\theta)(t)\|_{\dot H^{\frac32}} < +\infty,
  \end{equation}
  where $T^\ast>0$ is the supposed first blow-up time. Indeed, since
  the continuous embedding
  $\dot H^{3/2}(\mathbb{R}^3)\hookrightarrow
  \mathrm{BMO}(\mathbb{R}^3)$ holds, if \eqref{e:holds} is satisfied,
  then
  \begin{equation*}
    \sup_{0\leq t <T^\ast}\Pi(t)< +\infty,
  \end{equation*}
  and hence $\Pi(T^\ast) <+\infty$, which in turn implies, thanks to
  Theorem~\ref{th:bu criterion}, that such a blow-up time $T^\ast>0$
  cannot exist. Hence, global existence in time follows.
\end{remark}

In other words, the proof of Theorem \ref{main} reduces to the proof
of \eqref{e:holds}, provided the initial data are small enough.

First, we prove that the functions
$t \mapsto \|(u,v,\theta)(t)\|_{\dot H^{\frac12}}$ and
$t \mapsto \|(u,v,\theta)(t)\|_{\dot H^1}$ are decreasing, if the
initial data is small (see Corollary \ref{cor:monotone}). Then, with
this latter information in hand, we provide the uniform boundedness of
$t\mapsto \|(u,v,\theta)(t)\|_{\dot H^{\frac32}}$ as in
\eqref{e:holds} (see \eqref{e:final}).

\subsection{Monotonicity of $\dot H^{1}$- and $\dot H^{\frac12}$-norms under small initial data}

The key result that we are going to prove is Corollary \ref{cor:monotone}.
To this, we need to put together the energy estimates in $\dot H^1$ and in $\dot H^{\frac12}$,
as follows (see \eqref{stima-u-v-H1-chiusa} and \eqref{e:key2} below).

 \smallskip
  
\noindent \textbf{$\dot{H}^1$-estimates.}
Multiplying $\eqref{TCM-gen}_1$ by $-\Delta u$, integrating by parts,
we obtain
\begin{equation} \label{u-eq}
  \begin{aligned}
    \frac{1}{2} \frac{d}{dt} \|\nabla u(t)\|^2 + & \nu\|\Delta
    u(t)\|^2
    -  \sigma_1\int_{\R^3} u|u|^{\alpha-1}\cdot \Delta u\, dx \\
    & = \int_{\R^3}(u\cdot \nabla u)\cdot \Delta u \dx +\int_{\R^3}\dv
    (v\otimes v)\cdot \Delta u \dx .
  \end{aligned} \hspace{-0.5 cm}
\end{equation}

Multiplying $\eqref{TCM-gen}_2$ by $-\Delta v$, we obtain
\begin{equation} \label{v-eq}
  \begin{aligned}
    \frac{1}{2} \frac{d}{dt} \|\nabla v(t)\|^2 + &\eta\|\Delta
    v(t)\|^2
    -  \sigma_2\int_{\R^3} v|v|^{\beta-1}\cdot \Delta v\, dx \\
    & = \int_{\R^3} (u\cdot \nabla ) v\cdot \Delta v\dx + \int_{\R^3}
    (v \cdot \nabla) u \cdot \Delta v\dx
   + \int_{\R^3} \nabla \theta \cdot \Delta v \dx.
  \end{aligned}
\end{equation}

Taking the $L^2$-product of $\eqref{TCM-gen}_3$ with $-\Delta \theta$,
we find
\begin{equation*} \label{theta-eq} \frac{1}{2} \frac{d}{dt}
  \|\nabla\theta (t)\|^2 +\mu \|\Delta \theta\|^2 = \int_{\R^3}
  (u\cdot \nabla )\theta \cdot \Delta \theta \dx + \int_{\R^3} \dv v
  \, \Delta \theta\, dx.
\end{equation*}

Proceeding as in \cite{BBC-JMAA} (see also \cite{Yuan-Zhang,
  Yuan-Chen}), adding \eqref{u-eq} and \eqref{v-eq}, we obtain
\begin{equation} \label{stima-u-v-H1}
  \begin{aligned}
    \frac{1}{2}\frac{d}{dt} &\|\nabla (u,v,\theta)(t)\|^2 +
    \nu\|\Delta u\|^2 + \eta\|\Delta v\|^2 +\mu \|\Delta \theta\|^2+
    \sigma_1\||u|^{\frac{\alpha-1}{2}}\nabla u\|^2
    \\
    &\qquad + \frac{4\sigma_1(\alpha-1)}{(\alpha+1)^2}\|\nabla
    |u|^{\frac{\alpha+1}{2}}\|^2 +
    \sigma_2\||v|^{\frac{\beta-1}{2}}\nabla v\|^2 +
    \frac{4\sigma_2(\beta-1)}{(\beta+1)^2}\|\nabla
    |v|^{\frac{\beta+1}{2}}\|^2
    \\
    &\quad = \int_{\R^3}(u\cdot \nabla) u\cdot \Delta u \dx
    +\int_{\R^3}\dv (v\otimes v)\cdot \Delta u \dx + \int_{\R^3}
    (u\cdot \nabla ) v\cdot \Delta v\dx
    \\
    &\qquad + \int_{\R^3} (v \cdot \nabla) u \cdot \Delta v\dx +
    \int_{\R^3} \nabla \theta \cdot \Delta v \dx
    \\
    &\qquad +\int_{\R^3} (u\cdot \nabla )\theta \cdot \Delta \theta
    \dx + \int_{\R^3} \dv v \, \Delta \theta\, dx \doteq \sum_{i=1}^7
    J_i.
  \end{aligned}\hspace{-1 cm}
\end{equation}

Let us use $z$ and $w$ to represent $u$, $v$ or even $\theta$. By recalling
the embeddings $H^{\frac{1}{2}}(\mathbb R^3)\hookrightarrow L^3(\mathbb  R^3)$,
$H^1(\mathbb R^3)\hookrightarrow L^6(\mathbb R^3)$, and applying H\"older's, Young's,
Gagliardo--Nirenberg's inequalities, we reach 
\begin{equation*}
  \begin{aligned}
    J \doteq \int_{\R^3}|u||\nabla z| |\Delta w| \, dx &\le
    C\|u\|_3\|\nabla z\|_6 \|\Delta w\|
    \\
    &\le C\|u\|_3^2\|\nabla z\|_6^2 + \varepsilon \|\Delta w\|^2
    \\
    &\le C\|\Lambda^{\frac{1}{2}} u\|^2 \|\Delta z\|^2+\varepsilon
    \|\Delta w\|^2.
  \end{aligned}
\end{equation*}	

Proceeding as for $J$, we infer
\begin{gather*}
    J_1 \leq \int_{\R^3}|u||\nabla u| |\Delta u| \, dx \le
    C\|\Lambda^{\frac{1}{2}} u\|^2 \|\Delta u\|^2+\varepsilon \|\Delta
    u\|^2, \quad J_3 \leq \int_{\R^3}|u||\nabla v| |\Delta v|\, dx \le
    C\|\Lambda^{\frac{1}{2}} u\|^2 \|\Delta v\|^2+\varepsilon \|\Delta
    v\|^2,
    \\
  \textrm{ and }\,\,\,  J_4 \le \int_{\R^3}|v||\nabla u| |\Delta v| \, dx \le
    C\|\Lambda^{\frac{1}{2}} v\|^2 \|\Delta u\|^2+\varepsilon \|\Delta
    v\|^2.
\end{gather*}

Similarly, we have
\begin{equation*}
  \begin{aligned}
    &J_2 \le \int_{\R^3}|v|(|\nabla v|+|\div v|) |\Delta u| \, dx \le
    C\|\Lambda^{\frac{1}{2}} v\|^2 \|\Delta v\|^2+\varepsilon \|\Delta
    u\|^2,
    \\
     &J_6 \leq \int_{\R^3}|u||\nabla \theta| |\Delta \theta|\, dx \le
    C\|\Lambda^{\frac{1}{2}} u\|^2 \|\Delta \theta\|^2+\varepsilon
    \|\Delta \theta\|^2.
  \end{aligned}
\end{equation*}

Finally, we observe that
  \begin{equation*}
    J_5 + J_7 =\int_{\R^3} \nabla \theta \cdot \Delta v \dx +\int_{\R^3}
    \dv  v\,  \Delta \theta \dx =0.
  \end{equation*}

Plugging the above estimates into \eqref{stima-u-v-H1}, we obtain
\begin{equation}
  \label{e:H1}
  \begin{aligned}
    \frac{1}{2}\frac{d}{dt}\|\nabla (u,v, &\theta)(t)\|^2+
    \nu\|\Delta u\|^2 + \eta\|\Delta v\|^2 + \mu \|\Delta \theta\|^2 +
    {\frac{\sigma_1}{2}}\||u|^{\frac{\alpha-1}{2}}\nabla u\|^2
    \\
    & + \frac{4\sigma_1(\alpha-1)}{(\alpha+1)^2}\|\nabla
    |u|^{\frac{\alpha+1}{2}}\|^2
    +{\sigma_2}\||v|^{\frac{\beta-1}{2}}\nabla v\|^2 +
    \frac{4\sigma_2(\beta-1)}{(\beta+1)^2}\|\nabla
    |v|^{\frac{\beta+1}{2}}\|^2
    \\
    \le C & \big(\varepsilon + \|\Lambda^{\frac{1}{2}}
    (u,v,\theta)\|^2\big)\|\Delta (u,v,\theta)\|^2,
  \end{aligned}\hspace{-1 cm}
\end{equation}
and so
\begin{equation} \label{stima-u-v-H1-chiusa}
  \frac{d}{dt} \|\nabla (u,v, \theta)(t)\|^2 + \Big(2\min\{\nu, \eta, \nu\} -
  C\big(\varepsilon + \|\Lambda^{\frac{1}{2}}
  (u,v,\theta)\|^2\big)\Big)\|\Delta (u, v, \theta)\|^2 \leq 0.
\end{equation}
\begin{remark}
The solution, in $\dot H^{\frac12}$-norm, depends continuously on
the time variable, then there exists a positive time $T_1$, with
$0<T_1 < T^\ast$, such that the solution $(u, v, \theta)(t)$ satisfies
$C(\varepsilon +\|\Lambda^{\frac{1}{2}} (u,v,\theta)(t)\|^2)\leq
2\min\{\nu, \eta, \nu\}$, $t\in [0, T_1)$.  Later, in Lemma~\ref{lem:monotone}, we will
show that inequality \eqref{stima-u-v-H1-chiusa} holds true for every
$t=T_1$ with $0<t<T^*$.
\end{remark}

\noindent \textbf{$\dot{H}^{\frac{1}{2}}$-estimates.}
Now, we apply the $\Lambda^{\frac{1}{2}}$-operator to both sides of
$\eqref{TCM-gen}_1$, $\eqref{TCM-gen}_2$ and $\eqref{TCM-gen}_3$, and
we take the scalar product, respectively, with
$\Lambda^{\frac{1}{2}}u$, $\Lambda^{\frac{1}{2}}v$ and
$\Lambda^{\frac{1}{2}}\theta$, then
\begin{equation} \label{H1/2-stima-diff}
  \begin{aligned}
    \frac{1}{2} \frac{d}{dt} \|\Lambda^{\frac{1}{2}} & (u, v, \theta)
    (t) \|^2 + \nu\|\Lambda^{\frac{3}{2}} u\|^2 + \eta
    \|\Lambda^{\frac{3}{2}} v\|^2 + \mu \|\Lambda^{\frac{3}{2}}
    \theta\|^2
    \\
    = & -\int_{\R^3}\Lambda^{\frac{1}{2}}\big( (u\cdot \nabla) u\big)
    \cdot \Lambda^{\frac{1}{2}}u \dx -
    \int_{\R^3}\Lambda^{\frac{1}{2}}\big( (u\cdot \nabla) v\big) \cdot
    \Lambda^{\frac{1}{2}}v \dx
    \\
    & - \int_{\R^3}\Lambda^{\frac{1}{2}}\big( (u \cdot \nabla)
    \theta\big) \cdot \Lambda^{\frac{1}{2}}\theta \dx -
    \sigma_1\int_{\R^3}\Lambda^{\frac{1}{2}}(|u|^{\alpha-1}u)\cdot
    \Lambda^{\frac{1}{2}} u \dx\\
    & - \sigma_2\int_{\R^3}\Lambda^{\frac{1}{2}}(|v|^{\beta-1}v)\cdot
    \Lambda^{\frac{1}{2}} v \dx +\mathcal{K}\doteq \sum_{i=1}^5 K_i
    +\mathcal{K},
  \end{aligned}
\end{equation}
where
\begin{equation*}
  \begin{aligned}
    \mathcal{K}\doteq & -\int_{\R^3}\Lambda^{\frac{1}{2}}\big( (v\cdot
    \nabla) v\big) \cdot \Lambda^{\frac{1}{2}}u \dx -
    \int_{\R^3}\Lambda^{\frac{1}{2}}\big( (v\cdot \nabla) u\big) \cdot
    \Lambda^{\frac{1}{2}}v \dx
    \\
    & - \int_{\R^3}\Lambda^{\frac{1}{2}}\big( v \,\dv v\big)\cdot
    \Lambda^{\frac{1}{2}} u \dx - \int_{\R^3} \Lambda^{\frac{1}{2}}
    \dv v\cdot \Lambda^{\frac{1}{2}} \theta \dx - \int_{\R^3}
    \Lambda^{\frac{1}{2}} \nabla \theta\cdot \Lambda^{\frac{1}{2}} v
    \dx\doteq \sum_{j=1}^5\mathcal{K}_j,
  \end{aligned}
\end{equation*}
and observe that $\mathcal{K}_4+\mathcal{K}_5=0$
  since it coincides with the identity
  \begin{equation} \label{ultima-relazione-aggiunta}
 \int_{\R^3}  \Lambda^{\frac12} \dv v\cdot  \Lambda^{\frac12} \theta
 \dx + \int_{\R^3} \Lambda^{\frac12}\nabla \theta \cdot
 \Lambda^{\frac12} v \dx =0.
\end{equation}

By using H\"older's and Gagliardo--Nirenberg's inequalities, for
$i=1,2,3$ we find (here below $w$ may represent any of $u$, $v$ or $\theta$):
\begin{align} \label{K_i} &\begin{aligned} |K_i|&\leq \|(u \cdot
    \nabla) w \|_{\frac{3}{2}}\|\Lambda w\|_{3} \leq \|u\|_3 \|\nabla
    w\|_3 \|\Lambda w\|_{3}.
  \end{aligned}
\end{align}

For $K_4$, we separate the cases {$5/2 \leq \alpha \le 3$ and  $3 < \alpha <4$}.  
Consider first the case {$3 < \alpha <4$}, to get
  \begin{equation*} \label{K_4.ii}
    \begin{aligned}
      |K_4|&\leq \sigma_1 \|\Lambda^{-\frac{1}{2}}(|u|^{\alpha-1}u)\|
      \|\Lambda^{\frac{3}{2}}u\|
      \\
      &\leq C \||u|^{\alpha-1}u\|_{\frac{3}{2}}
      \|\Lambda^{\frac{3}{2}}u\|
      \\
      & =
      C\|u\|_{\frac{3\alpha}{2}}^{\alpha}\|\Lambda^{\frac{3}{2}}u\|
      \\
            & {\leq C\|\Lambda^\frac12 u\|^{4-\alpha}\, \|\nabla
        u\|^{2(\alpha-2)} \|\Lambda^\frac32 u\|}
      \\
      & {\leq C\|\Lambda^\frac12 u\|^{4-\alpha}\, \|\nabla
        u\|^{2(\alpha-2)-2}\|\nabla u\|^2 \|\Lambda^\frac32 u\|}
      \\
      & {\leq C\|\Lambda^\frac12 u\|^{4-\alpha}\, \|\nabla
        u\|^{2(\alpha-3)}(\|\Lambda^{\frac12} u\|\, \|\Lambda^{\frac32}
        u\|)\,\|\Lambda^\frac32 u\|}
      \\
      & { \leq C\|\Lambda^\frac12 u\|^{5-\alpha}\, \|\nabla
        u\|^{2(\alpha-3)} \,\|\Lambda^\frac32 u\|^2},
    \end{aligned}
  \end{equation*}
 where we used the Gagliardo--Nirenberg's inequality, {and $1<\alpha -2 <2$},
\begin{equation*}
  \|u\|_{\frac{3\alpha}{2}}\leq \|\Lambda^{\frac{1}{2}} u\|^\frac{4-\alpha}{\alpha}
  \|\nabla u\|^{\frac{2(\alpha-2)}{\alpha}}. 
\end{equation*}

\smallskip

  Consider now $K_4$ when {$5/2 \leq \alpha \le 3$}. In
  this case (at least, for $\alpha<3$), we cannot proceed directly as before, since
  {$2(\alpha-2) \le 2$} and hence {$2(\alpha-2)-2\le 0$}. Therefore, we
  manipulate the integral defining $K_4$, in order to make use of
  commutator estimates \eqref{KPV}.  In particular, we take
  advantage of
  \begin{equation} \label{utility-K4}
    \begin{aligned}
      K_4 &=-\sigma_1 \int_{\R^3}\Lambda^{\frac{1}{2}}(|u|^{\alpha-1}u)\cdot
      \Lambda^{\frac{1}{2}} u \dx
      \\
      &=-\sigma_1 \int_{\mathbb R^3}|u|^{\alpha-1}\, \Lambda^\frac12
      u\cdot \Lambda^\frac12 u\, dx -
      \sigma_1\int_{\R^3}\Big(\Lambda^{\frac{1}{2}}(|u|^{\alpha-1}u) -
      |u|^{\alpha-1}\Lambda^\frac12 u \Big)\cdot \Lambda^\frac12 u\dx
      \\
      &=-\sigma_1\||u|^{\frac{\alpha-1}{2}}\, \Lambda^\frac12 u\|^2
      -\sigma_1\int_{\R^3}\Big(\Lambda^{\frac{1}{2}}(|u|^{\alpha-1}u)
      - |u|^{\alpha-1}\Lambda^\frac12 u \Big)\cdot \Lambda^\frac12
      u\dx
      \\
      &\doteq -\sigma_1\||u|^{\frac{\alpha-1}{2}}\, \Lambda^\frac12 u\|^2 -
      \mathcal I,
    \end{aligned}
  \end{equation}
  and $-\sigma_1\||u|^{\frac{\alpha-1}{2}}\, \Lambda^\frac12 u\|^2$
  will be moved to the left-hand side of \eqref{H1/2-stima-diff}.

  Hence, we need to estimate $\mathcal I$, for which we can use
  \eqref{KPV}, as follows. We have
  \begin{equation*}
    \begin{aligned}
      |\mathcal I| &\le C\|\Lambda^\frac12\big(|u|^{\alpha-1}\,
      u\big)-|u|^{\alpha-1}\, \Lambda^\frac12 u\|_{H^{-1}}\,
      \|\Lambda^\frac32 u\|
      \\
      & \le C\|\Lambda^\frac12\big(|u|^{\alpha-1}\,
      u\big)-|u|^{\alpha-1}\, \Lambda^\frac12 u\|_{\frac65}\,
      \|\Lambda^\frac32 u\|,
    \end{aligned}
  \end{equation*}
  where we used duality and the continuous embedding
  $L^\frac{6}{5}(\mathbb R^3) \hookrightarrow \dot H^{-1}(\mathbb R^3)$ (see, e.g., \cite{Zhao}). To
  the first integral on the last line above, we can apply \eqref{KPV}
  with $p=6/5$, $p_1= \frac{3}{\alpha-1}$ and
  $q_1= \frac{6}{7-2\alpha}$ and $s=1/2$. {Hence, by making use of the continuous embeddings
  $H^\frac12(\mathbb R^3) \hookrightarrow L^3(\mathbb R^3)$,
  $W^{\frac12, \frac{3}{\alpha-1}}(\mathbb R^3) \hookrightarrow
  L^{\frac{6}{2\alpha-3}}(\mathbb R^3)$ and also $W^{\frac12, \frac{3}{4-\alpha}}(\mathbb R^3) \hookrightarrow
  L^{\frac{6}{7-2\alpha}}(\mathbb R^3)$ (see \cite[Theorem 6.5]{DiNezza}, the related embeddings and estimates),
  and since the latter one, by duality, implies that $L^\frac{3}{4-\alpha}(\mathbb R^3)
  \hookrightarrow W^{-\frac12, \frac{6}{7-2\alpha}}(\mathbb R^3) $, we obtain the following chain of inequalities}
  \begin{equation*}
    \begin{aligned}
      |\mathcal I| &\le C \Big( \|\nabla
      |u|^{\alpha-1}\|_{\frac{3}{\alpha-1}}\,
      \|\Lambda^{-\frac{1}{2}}u\|_{\frac{6}{7-2\alpha}}+\|u\|_{p_2}\,
      \|\Lambda^{\frac12} |u|^{\alpha-1}\|_{q_2}\Big)\,
      \|\Lambda^\frac32 u\|
      \\
      &\le C\||u|^{\alpha-2}\, \nabla u\|_{\frac{3}{\alpha-1}}\,
      \|\Lambda^{-\frac{1}{2}} u\|_{\frac{6}{7-2\alpha}}\,
      \|\Lambda^{\frac32} u\| + C\,
      \|u\|_{\frac{3}{4-\alpha}}\|\,\|\Lambda^{\frac12}
      |u|^{\alpha-1}\|_{\frac{6}{2\alpha-3}}\|\,
      \|\Lambda^{\frac32} u\|
      \\
      &\le C\big(\||u|^{\alpha-2}\|_{\frac{3}{\alpha-2}}\, \|\nabla u\|_3\big)\,
      \|\Lambda^{-\frac{1}{2}} u\|_{\frac{6}{7-2\alpha}}\,
      \|\Lambda^{\frac32} u\|+C\, \|u\|_{\frac{3}{4-\alpha}}\|\,\|\nabla
      |u|^{\alpha-1}\|_{\frac{3}{\alpha-1}}\|\,
      \|\Lambda^{\frac32} u\|
      \\
      &\le C \|u\|^{\alpha-2}_{\frac{3(\alpha-2)}{\alpha-2}}\,
      \|\Lambda^\frac32 u\|\, \|\Lambda^{-\frac{1}{2}}
      u\|_{\frac{6}{7-2\alpha}}\, \|\Lambda^\frac32 u\|+C\,
      \|u\|_{\frac{3}{4-\alpha}}\,\|\Lambda^\frac12 u\|^{\alpha-2}\,
      \|\Lambda^\frac32 u\|^2
      \\
      &\le C\|u\|_3^{\alpha-2}\,\|u\|_{\frac{3}{4-\alpha}}\,
      \|\Lambda^\frac32 u\|^2+C\,
      \|u\|_{\frac{3}{4-\alpha}}\,\|\Lambda^\frac12 u\|^{\alpha-2}\,
      \|\Lambda^\frac32 u\|^2
      \\
      &\le C\|\Lambda^\frac12
      u\|^{\alpha-2}\,\|u\|_{\frac{3}{4-\alpha}}\, \|\Lambda^\frac32
      u\|^2+C\, \|u\|_{\frac{3}{4-\alpha}}\,\|\Lambda^\frac12
      u\|^{\alpha-2}\, \|\Lambda^\frac32 u\|^2.
    \end{aligned}
  \end{equation*}
  Since {$5/2 \le \alpha \le 3$, we have
  $2\le \bar p\doteq  3/(4-\alpha) \le 3$}, so that
  $\|u\|_{\bar p} \le C\|u\|^{1-\eta}\, \|\Lambda^\frac12 u\|^\eta$,
  with $\eta = 3(\bar p-2)/\bar p = 2\alpha-5$.
  Hence, we obtain
  \begin{equation*}
    |\mathcal I|\le C\, \|\Lambda^{\frac12} u\|^{3\alpha -7}\,\|\Lambda^{\frac32} u\|^2.
  \end{equation*}
  
  Therefore, we have proved that
      \begin{equation*} \label{K_4.ii} \left\{
        \begin{aligned}
          & |K_4| \le C\,\|\Lambda^{\frac12} u\|^{5-\alpha}\, \|\nabla
          u\|^{2(\alpha-3)}\, \|\Lambda^{\frac32} u\|^2, \,\,\,\, &&\mbox{
            if } \ {3< \alpha <4},
          \\
           & |K_4+\sigma_1\||u|^\frac{\alpha-1}{2}\, \nabla u\|^2|= |\mathcal{I}|\le C\,
          \|\Lambda^\frac12 u\|^{3\alpha -7}\,\|\Lambda^{\frac32} u\|^2,
          \,\,\,\, && \mbox{ if } \ {5/2 \le \alpha \le 3}.
        \end{aligned}
      \right.
    \end{equation*}

Similarly, for $K_5$, we have
    \begin{equation*} \label{K_5.ii} \left\{
        \begin{aligned}
          & |K_5| \le C\,\|\Lambda^\frac12 v\|^{5-\beta}\, \|\nabla
          v\|^{2(\beta-3)}\, \|\Lambda^\frac32 v\|^2, \,\,\,\, &&\mbox{
            if } \ {3< \beta <4},
          \\
           & |K_5+\sigma_2 \||v|^\frac{\beta-1}{2}\, \nabla v\|^2| \le C\,
          \|\Lambda^\frac12 v\|^{3\beta -7}\,\|\Lambda^\frac32 v\|^2,
          \,\,\,\, && \mbox{ if } \ {5/2 \le \beta \le 3}.
        \end{aligned}
      \right.
    \end{equation*}

Now, consider $\mathcal{K}= \sum_{j=1}^3\mathcal{K}_j$ and, in particular, we have that
\begin{equation*}
  \begin{aligned}
    |\mathcal{K}_1| + |\mathcal{K}_2| +
    |\mathcal{K}_3|
    & \leq \|\big( (v\cdot \nabla) v\big)\|_{\frac{3}{2}}\|\Lambda u
    \|_3+ \|\big( (v\cdot \nabla) u\big)\|_{\frac{3}{2}} \|\Lambda v
    \|_3 + \| v\dv v \|_{\frac{3}{2}} \|\Lambda u \|_3
    \\
    & \leq C \|\Lambda^{\frac{1}{2}} v\| \|\Lambda^{\frac{3}{2}}
    u\|\|\Lambda^{\frac{3}{2}}v\|\\
    &\leq C \|\Lambda^{\frac{1}{2}} v\| \big( \|\Lambda^{\frac{3}{2}}
    u\|^2 +\|\Lambda^{\frac{3}{2}}v\|^2\big).
  \end{aligned}
\end{equation*}

  We put together all the previous terms, unifying in a single estimate,
  what has been obtained in the two sub-cases of interest, i.e. {$5/2 \le \alpha \le 3$} and {$3 < \alpha < 4$}.
  We argue as in \cite{Chae-Lee, Zhang-Xu}, but as a consequence of the four couples of
estimates coming from $K_4$ and $K_5$, 
we have that
\begin{equation}\label{H1/2-stima-diff.ii}
  \begin{aligned}
    \frac{1}{2} \frac{d}{dt} & \|\Lambda^{\frac{1}{2}} (u, v,
    \theta) (t) \|^2 + \min\{\nu, \eta, \mu\}
    \|\Lambda^{\frac{3}{2}} (u, v, \theta)\|^2 + \Phi(u,\alpha)+ \tilde \Phi (v,\beta)
    \\
    &\leq C\big(\|\Lambda^{\frac{1}{2}} (u, v, \theta)  \| + \mathcal R(u,\alpha) + \tilde{\mathcal R}(v,\beta)
    \big) \|\Lambda^{\frac{3}{2}} (u, v, \theta)\|^2,
  \end{aligned}
\end{equation}
where
\begin{equation*}
\begin{aligned}
 \Phi(u,\alpha)= 
\left\{
\begin{aligned}
  &0, \ &&\mbox{ if } \ {3< \alpha <4},\\
  &\sigma_1\, \||u|^\frac{\alpha-1}{2}\, \Lambda^{\frac12 }u\|^2, \ &&\mbox{ if } \ {5/2 \le \alpha \le 3},
\end{aligned}
 \right.
 \qquad
\tilde{\Phi}(v,\beta)= 
\left\{
\begin{aligned} 
  &0, \ &&\mbox{ if } \ {3< \beta <4}\\
  &\sigma_2\, \||v|^\frac{\beta-1}{2}\, \Lambda^{\frac12}v\|^2, \ &&\mbox{ if } \ {5/2 \le \beta \le 3}, 
\end{aligned}
 \right.
\end{aligned}
\end{equation*}
and
\begin{multline*}
\mathcal R(u, \alpha)\le \left\{
        \begin{aligned}
          &  C\,\|\Lambda^{\frac12}u\|^{5-\alpha}\, \|\nabla
          u\|^{2(\alpha-3)}, \,\, &&\mbox{
            if } \ {3 < \alpha <4},
          \\
           &  C\,
          \|\Lambda^{\frac12} u\|^{3\alpha -7},
          \,\,\,\, && \mbox{ if } \ {5/2 \le \alpha \le 3},
        \end{aligned}
      \right.
\\
      \tilde{\mathcal R}(v, \beta)\le \left\{
        \begin{aligned}
          &  C\,\|\Lambda^{\frac12} v\|^{5-\beta}\, \|\nabla
          v\|^{2(\beta-3)}, \,\, && \mbox{
            if } \ {3< \beta <4},
          \\
           &  C\,
          \|\Lambda^{\frac12}v\|^{3\beta -7},
          \,\,\,\, && \mbox{ if } \ {5/2 \le \beta \le 3}.
        \end{aligned}
      \right.
\end{multline*}
Recalling that $\ell = \min \{\nu, \eta, \mu\}$ as introduced in \eqref{e:local0}, we have
\begin{equation} \label{e:key2}
  \begin{aligned}
    \frac{1}{2} \frac{d}{dt}  \|\Lambda^{\frac{1}{2}} & (u, v,
    \theta) (t) \|^2  +  \Phi(u,\alpha)+ \tilde \Phi (v,\beta)
    \\
    & + \big(\ell - C \|\Lambda^{\frac{1}{2}} (u, v, \theta) (t) \| -
    C\,\mathcal R(u,\alpha) -C\, \tilde{\mathcal R}(v,\beta)\big) \|\Lambda^{\frac{3}{2}} (u, v,
    \theta)\|^2 \leq 0.
  \end{aligned}
\end{equation}

From \eqref{stima-u-v-H1-chiusa} and \eqref{e:key2}, it follows the
following key lemma.

\begin{lem}[Monotonicity of $\dot H^1$- and $\dot H^{\frac12}$-norms]\label{lem:monotone}
  Assume $5/2 \le \alpha, \beta <4$. Let $(u_0, v_0, \theta_0)\in H^2(\mathbb{R}^3)^3$,
  with $\dv u_0 =0$, be such that
   \begin{equation} \label{e:keyii1}
   \ell - C_1\, \big(\|\Lambda^{\frac{1}{2}} (u_0, v_0, \theta_0)  \| -
    \mathcal R(u_0,\alpha) - \tilde{\mathcal R}(v_0,\beta) \big) >0,
    \end{equation}
    and
    \begin{equation} \label{e:keyii2} 2\, \ell - C_2\,\big(\varepsilon
      + \|\Lambda^{\frac{1}{2}} (u_0,v_0,\theta_0)\|^2\big) >0,
    \end{equation}
    where $C_1, C_2>0$ are the constants $C$ in
    \eqref{stima-u-v-H1-chiusa} and \eqref{e:key2}, respectively.
  
    Further, let $(u,v,\theta)(t)$ be the local strong solution of
    \eqref{TCM-gen} and let $T^\ast>0$ be  the first blow-up
    time, as defined in Theorem \ref{th:bu criterion}.
    Then, the following relations hold true
    \begin{equation}
      \label{e:monotonicity1}
      \sup_{0\leq t <T^\ast} \|(u,v,\theta)(t)\|_{\dot H^{\frac12}}
      \le \|(u_0,v_0,\theta_0)\|_{\dot H^{\frac12}}, \ \ \ \
      \sup_{0\leq t <T^\ast} \|(u,v,\theta)(t)\|_{\dot H^1} \le \|(u_0,v_0,\theta_0)\|_{\dot H^1},
    \end{equation}
    and moreover
    \begin{equation}
      \label{e:integrability1}
      \int_{0}^{T^\ast} \|\Delta (u,v,\theta)(t)\|^2\, dt < +\infty.
    \end{equation}
  \end{lem}

  Observe that when {$5/2 \le \alpha, \beta \le 3$}, the quantity in
  between the parentheses in \eqref{e:keyii1} is only a function of
  $\|\Lambda^\frac12 (u_0,v_0,\theta_0)\|$, while when $3<\alpha<4$ or
  $3< \beta <4$ $\|\nabla u_0\|$ and/or $\|\nabla v_0\|$ also appear.

  \smallskip

  Thus, we have the following consequence of Lemma~\ref{lem:monotone}.

\begin{cor}
\label{cor:monotone}
  Under the assumptions of Lemma~\ref{lem:monotone}, we have:
  \begin{enumerate}[(i)]
  \item When {$5/2 \le \alpha, \beta \le 3$}, there exists $c_0>0$ such
    that if $\|\Lambda^\frac12 (u_0,v_0,\theta_0)\| \le c_0$, then
    \eqref{e:monotonicity1} and \eqref{e:integrability1} hold
    true.\\[-0.3 cm]

  \item In the other cases, if $\|(u_0, v_0, \theta_0)\|_{H^2} \le R$,
    for some $R>0$, then there exists $c_0(R)>0$ such that if
    $\|\Lambda^\frac12 (u_0,v_0,\theta_0)\| \le c_0(R)$, then
    \eqref{e:monotonicity1} and \eqref{e:integrability1} hold true.
  \end{enumerate}
\end{cor}

Observe that $c_0(R) =o(1)$ if $R\to +\infty$.

\begin{proof}[Proof of Lemma~\ref{lem:monotone}]
  From \eqref{e:keyii1}--\eqref{e:keyii2} and the continuous
  dependence in time with respect to the initial data,
  \eqref{stima-u-v-H1-chiusa} and \eqref{e:key2} imply that
  \begin{equation}
    \label{e:mono} \frac{d}{dt}\| \nabla (u,v,\theta)
    (t)\|^2\leq 0\,\,\, \textrm{ and }\,\,\, \frac{d}{dt} \|
    \Lambda^{\frac12} (u,v,\theta)(t)\|^2 \le 0,
  \end{equation}
  for every $0<t<T_1$, for some $T_1>0$. This implies that
  \eqref{e:keyii1} and \eqref{e:keyii2} hold, for every $0<t<T_1$. We
  claim that \eqref{e:mono} must hold true for every $0<t<T^*$, where
  $T^\ast$ is the (supposed) first blow-up time. Assume then, by
  contradiction, that there exists $T_2>0$ with $T_1 < T_2 <T^*$ such
  that
  \begin{equation*}
    \frac{d}{dt}\|\nabla (u,v,\theta)(T_2)\|^2 >0.
  \end{equation*}
  The case in which is $\frac{d}{dt}\|\Lambda^{\frac{1}{2}} (u,v,\theta)(T_2)\|^2 >0$ can be
  treated similarly.

  Now, set
  \begin{equation*}
    \tilde{T}\doteq \inf\left\{\tau >T_1: \, \frac{d}{dt}\|\nabla (u,v,\theta)(\tau)\|^2 >0\right\}.
  \end{equation*}
  Clearly, $\tilde{T}$ is such that for every $0<t<\tilde{T}$,
  \eqref{e:mono} holds true. This implies that there exists a
  neighborhood of $\tilde{T}$, say
  $(\tilde{T}-\delta, \tilde{T}+\delta)$ such that \eqref{e:keyii1}
  and \eqref{e:keyii2} are true. By using \eqref{stima-u-v-H1-chiusa}
  and \eqref{e:key2}, we deduce that also \eqref{e:mono} is true, in
  $(\tilde{T}-\delta, \tilde{T}+\delta)$. This provides a
  contradiction. Hence, we proved \eqref{e:monotonicity1}.
  
  By integrating \eqref{stima-u-v-H1-chiusa}, with respect to time
  variable $t\in [0, T^\ast)$, we get also \eqref{e:integrability1}.
\end{proof}

\subsection{Last step: Boundedness in $\dot H^{\frac32}$-norm}

  As explained in Remark~\ref{rmk-Linfty-H3/2},
  to conclude the proof of Theorem~\ref{main}
  we resort to the energy estimate in the $\dot{H}^{\frac{3}{2}}$-norm.

\smallskip

\noindent \textbf{$\dot{H}^{\frac{3}{2}}$-estimates.}
Taking the operator $\Lambda^{\frac{3}{2}}$ on both sides of
$\eqref{TCM-gen}_1$ and $\eqref{TCM-gen}_2$, and considering the scalar
product with $\Lambda^{\frac{3}{2}} u$ and $\Lambda^{\frac{3}{2}} v$,
respectively, then proceeding in the same way for the third equation
$\eqref{TCM-gen}_3$, and summing up the resulting equations, we deduce
that
\begin{equation} \label{H3/2-stima-diff}
  \begin{aligned}
    \frac{1}{2} \frac{d}{dt} & \|\Lambda^{\frac{3}{2}} (u, v, \theta)
    (t) \|^2 + \min\{\nu, \eta, \mu\} \|\Lambda^{\frac{5}{2}} (u, v,
    \theta)\|^2
    \\
    = & -\int_{\R^3}\Lambda^{\frac{3}{2}}\big( (u\cdot \nabla) u\big)
    \cdot \Lambda^{\frac{3}{2}}u \dx -
    \int_{\R^3}\Lambda^{\frac{3}{2}}\big( (u\cdot \nabla) v\big) \cdot
    \Lambda^{\frac{3}{2}}v \dx - \int_{\R^3}\Lambda^{\frac{3}{2}}\big(
    (u \cdot \nabla) \theta\big) \cdot \Lambda^{\frac{3}{2}}\theta \dx
    \\
    &\qquad \qquad -
    \sigma_1\int_{\R^3}\Lambda^{\frac{3}{2}}(|u|^{\alpha-1}u)\cdot
    \Lambda^{\frac{3}{2}} u \dx -
    \sigma_2\int_{\R^3}\Lambda^{\frac{3}{2}}(|v|^{\beta-1}v)\cdot
    \Lambda^{\frac{3}{2}} v \dx +\mathcal{D},
  \end{aligned}
\end{equation}
where
\begin{equation*}
  \mathcal{D}\doteq -\int_{\R^3}\Lambda^{\frac{3}{2}}\big( (v\cdot \nabla) v\big) \cdot
  \Lambda^{\frac{3}{2}}u \dx - \int_{\R^3}\Lambda^{\frac{3}{2}}\big( (v\cdot \nabla) u\big)
  \cdot \Lambda^{\frac{3}{2}}v \dx - \int_{\R^3}\Lambda^{\frac{3}{2}}\big( v \,\dv v\big)
  \cdot \Lambda^{\frac{3}{2}} u \dx \doteq \sum_{i=1}^{3} \mathcal{D}_i,
\end{equation*}
and we used the identity equivalent to 
\eqref{ultima-relazione-aggiunta}, but with $\Lambda^{\frac32}$ in place of $\Lambda^{\frac12}$.

Then, considering $w=u,v,\theta$, from \eqref{KP}, the embeddings
$H^{\frac{1}{2}}\hookrightarrow L^3(\R^3)$ and
$H^1\hookrightarrow L^6(\R^3)$, and Gagliardo--Nirenberg's inequalities
$\|\Lambda^\frac32 w\|_q \le C\|\nabla w\|^{1-\delta}
\|\Lambda^\frac52 w\|^{\delta}$with $\delta=\frac{2(2q-3)}{3q}$,
$\frac{3}{2}\leq q\leq 6$, and
$\|u\|_\frac{10}{3}\le C \|u\|^\frac{2}{5} \|\nabla u\|^\frac{3}{5}$,
we have
\begin{equation*}
  \begin{aligned}
    \| \Lambda^{\frac{1}{2}}\big( (u\cdot \nabla) w\big) \|
    \|\Lambda^{\frac{5}{2}}w\| &\leq
    C\|u\|_\frac{10}{3}\|\Lambda^{\frac{1}{2}}\nabla w\|_5
    \|\Lambda^{\frac{5}{2}}w\| + C\|\Lambda^{\frac{1}{2}}
    u\|_3\|\nabla w\|_6 \|\Lambda^{\frac{5}{2}}w\|
    \\
    & \leq C \|\nabla u\|^{\frac{3}{5}} \big(\|\nabla
    w\|^{\frac{1}{15}} \|\Lambda^{\frac{5}{2}} w\|^{\frac{14}{15}}
    \big) \|\Lambda^{\frac{5}{2}} u\|+ C\|\nabla u\|\|\Delta w\|
    \|\Lambda^{\frac{5}{2}}w\|
    \\
    & \leq C \|\nabla u\|^\ast\|\nabla w\|^\ast+\varepsilon
    \|\Lambda^{\frac{5}{2}}w\|^2 + C\|\nabla u\|\|\Delta w\|
    \|\Lambda^{\frac{5}{2}}w\|
    \\
    & \leq C \|\nabla u\|^\ast\|\nabla w\|^\ast+\varepsilon
    \|\Lambda^{\frac{5}{2}}w\|^2 + C\|\nabla
    u\|\|\Lambda^{\frac{1}{2}} w\|^{\frac{1}{4}}
    \|\Lambda^{\frac{5}{2}}w\|^{\frac{7}{4}}
    \\
    &\le C \|\nabla (u, v, \theta)\|^\ast + \varepsilon
    \|\Lambda^{\frac{5}{2}} (u, v, \theta)\|^2.
  \end{aligned}
\end{equation*}
  
From here above
on, we make no explicit the exponents of
$\|\nabla (u,v,\theta)\|^\ast$ and
$\|\Lambda^\frac12 (u,v,\theta)\|^\ast$, since as a consequence of
Lemma \ref{lem:monotone}, they are
bounded by a constant
depending on $\|\nabla (u_0,v_0,\theta_0)\|$ and
$\|\Lambda^\frac12 (u_0,v_0, \theta_0)\|$.

Analogously, about the first and third addendum of $\mathcal{D}$, we
infer
\begin{equation*}
  \begin{aligned}
    |\mathcal{D}_1+\mathcal{D}_3| &\le \big(\|
    \Lambda^{\frac{1}{2}}\big( (v\cdot \nabla) v\big) \|
    +\|\Lambda^{\frac{1}{2}}\big( v \,\dv v\big)\|\big) \|
    \Lambda^{\frac{5}{2}}u\|
    \\
    & \leq C\big(\|v\|_{\frac{10}{3}}\|\Lambda^{\frac{1}{2}}\nabla
    v\|_5 + \|\Lambda^{\frac{1}{2}}v\|_3\|\nabla v\|_6
    \big)\|\Lambda^{\frac{5}{2}}u\|
    \\
    &\le C \|\nabla (u, v, \theta)\|^\ast+ \varepsilon\big(
    \|\Lambda^\frac52 u\|^2+\|\Lambda^\frac52 v\|^2\big).
  \end{aligned}
\end{equation*}
Following the same arguments, switching the role of $u$ and $v$ where
necessary, we immediately get
\begin{equation*}
  |\mathcal{D}_2| 
  \le C \|\nabla (u, v, \theta)\|^\ast 
  +  \varepsilon\big( \|\Lambda^\frac52 u\|^2+\|\Lambda^\frac52 v\|^2\big).
\end{equation*}
Let us now consider the term
$\int_{\R^3}\Lambda^{\frac{3}{2}}(|u|^{\alpha-1}u)\cdot
\Lambda^{\frac{3}{2}} u \dx$. We need to separate the case
$5/2\le\alpha <3$ from  the case $3\le \alpha <4$.

\begin{equation*}
  \begin{aligned}
    \left| \int_{\R^3}\Lambda^{\frac{3}{2}}(|u|^{\alpha-1}u)\cdot
      \Lambda^{\frac{3}{2}} u \dx \right| &\leq C \|
    \Lambda^{\frac{1}{2}}\big(|u|^{\alpha-1} u\big) \|
    \|\Lambda^{\frac{5}{2}}u\|
    \\
    &\leq C\big(\||u|^{\alpha-1}\|_p
    \|\Lambda^{\frac{1}{2}}u\|_{\frac{2p}{p-2}} +
    \|\Lambda^{\frac{1}{2}}|u|^{\alpha-1}\|\|u\|_\infty \big)
    \|\Lambda^{\frac{5}{2}}u\|.
  \end{aligned}
\end{equation*}
About the first addendum, we have
\begin{equation*}
  \begin{aligned}
    \||u|^{\alpha-1}\|_p
    \|\Lambda^{\frac{1}{2}}u\|_{\frac{2p}{p-2}}\,\|\Lambda^{\frac{5}{2}}u\|&\le
    C\,\|u\|^{\alpha-1}_{p(\alpha-1)}\,\|\Lambda^\frac12
    u\|^{\frac{p-6}{p}}\,\|\nabla u\|^{\frac{6}{p}}\,
    \|\Lambda^{\frac{5}{2}}u\|
    \\
    & \leq C \,\|\Lambda^\frac12 u\|^{\frac{p-6}{p}}\,\|\nabla
    u\|^{*}\, \|\Lambda^\frac52 u\|^{1+\delta(p)}
    \\
    & \leq C\,\|\nabla u\|^{*} \,\|\Lambda^\frac12
    u\|^{*}+\varepsilon\|\Lambda^{\frac{5}{2}}u\|^2,
  \end{aligned}
\end{equation*}
where we make use of $p>6$, that implies both
\begin{equation*}
  \begin{aligned}
    \|\Lambda^\frac12 u\|_{\frac{2p}{p-2}}&\le C\,\|\Lambda^\frac12
    u\|^{\frac{p-6}{p}}\,\|\nabla u\|^{\frac{6}{p}},
    \\
    \|u\|_{p(\alpha-1)}^{\alpha-1} &\le C\,\|u\|^{*} \|\Lambda^\frac52
    u\|^{\delta(p)} \,\,\,\, \mbox{ with } \
    \delta(p)=\frac{3p(\alpha-1)-6}{5p}.
  \end{aligned}
\end{equation*}

  About the second addendum, we have to separate
  the case $5/2\leq \alpha <3$ from the case $3 \le \alpha <4$.
      
\noindent {\em Case $5/2\leq\alpha\le 3$:}
\begin{equation*}
  \begin{aligned}
    \|\Lambda^{\frac{1}{2}}\,|u|^{\alpha-1}\|\|u\|_\infty \,
    \|\Lambda^{\frac{5}{2}}u\|&\le C\,\|
    |u|^{\alpha-1}\|^{\frac{1}{2}} \|\nabla
    |u|^{\alpha-1}\|^{\frac{1}{2}} \,\|\nabla
    u\|^\frac23\,\|\Lambda^\frac52 u\|^\frac13\,
    \|\Lambda^{\frac{5}{2}}u\|
    \\
    & \leq C\, \|u\|_{2(\alpha-1)}^\frac{\alpha-1}{2}\, \|\nabla
    |u|^{\alpha-1}\|^\frac12\, \|\nabla u\|^*\, \|\Lambda^\frac52
    u\|^\frac{4}{3}
    \\
    & \leq C\, \|\nabla |u|^{\alpha-1}\|^\frac32\, \|\nabla u\|^*+
    \varepsilon \|\Lambda^\frac52 u\|^2
    \\
    & \leq C\, \|\nabla u\|^*+ \varepsilon \|\nabla
    |u|^{\alpha-1}\|^2+ \varepsilon \|\Lambda^\frac52 u\|^2
    \\
    & = C\,\big(1+ \|\nabla u\|^*\big)+ \varepsilon
    \||u|^{\frac{\alpha-1}{2}}\, \nabla u\|^2+ \varepsilon
    \|\Lambda^\frac52 u\|^2,
  \end{aligned}
\end{equation*}
where we applied the following Gagliardo--Nirenberg's inequalities
\begin{equation*}
  \begin{aligned}
    \|u\|_\infty &\le C\, \|\nabla u\|^\frac23\, \|\Lambda^\frac52
    u\|^\frac13,
    \\
    \|u\|_{2(\alpha-1)}&\leq C \|u\|^{1-\delta}\|\nabla
    u\|^\delta,\,\, \textrm{ with }\,\, \delta =
    \frac{3(\alpha-2)}{2(\alpha-1)},
  \end{aligned}
\end{equation*}
together with the next manipulation of the term involving
$\nabla |u|^{\alpha-1}$, i.e.
\begin{equation*}
  \begin{aligned}
    \|\nabla |u|^{\alpha-1}\|^2 &= \int_{\mathbb{R}^3} \big|\nabla
    |u|^{\alpha-1}\big|^2\,dx
    \\
    &= (\alpha-1)\int_{\mathbb{R}^3}|u|^{2(\alpha-2)}\,|\nabla u|^2\,
    dx
    \\
    & =C\int_{\mathbb{R}^3} \big(|u|^{2(\alpha-2)}\, |\nabla
    u|^{\frac{4(\alpha-2)}{\alpha-1}}\big)\,|\nabla
    u|^\frac{6-2\alpha}{\alpha-1}\,dx
    \\
    &\le \varepsilon\int_{\mathbb{R}^3}|u|^{\alpha-1}\, |\nabla
    u|^2\,dx + C\int_{\mathbb{R}^3} |\nabla u|^{2}\, dx
    \\
    &=\varepsilon \||u|^{\frac{\alpha-1}{2}}\nabla u\|^2 + C\,
    \|\nabla u\|^2.
  \end{aligned}
\end{equation*}

\noindent {\em Case $3< \alpha <4$:}
\begin{equation*}
  \begin{aligned}
    \|\Lambda^{\frac{1}{2}}\,|u|^{\alpha-1}\|\|u\|_\infty \,
    \|\Lambda^{\frac{5}{2}}u\|&\le C\, \|
    |u|^{\alpha-1}\|^{\frac{1}{2}} \|\nabla
    |u|^{\alpha-1}\|^{\frac{1}{2}}\, \| u\|_{\infty}\,
    \|\Lambda^{\frac{5}{2}}u\|
    \\
    &\le (\alpha-1)C\,\|u\|_{2(\alpha-1)}^{\frac{\alpha-1}{2}} \|
    |u|^{\alpha-2}\nabla u\|^{\frac{1}{2}} \,\big(\|\nabla
    u\|^{\frac{2}{3}}\|\Lambda^{\frac{5}{2}}u\|^{\frac{1}{3}}\big)\,
    \|\Lambda^{\frac{5}{2}}u\|
    \\
    &\le C\,\|\nabla u\|^{\frac{3(\alpha-2)}{4}+
      \frac{2}{3}}\big(\||u|^{\alpha-2}\|_3 \,\|\nabla
    u\|_6\big)^{\frac{1}{2}}\,
    \|\Lambda^{\frac{5}{2}}u\|^{\frac{4}{3}}
    \\
    &\le C\,\|\Lambda^{\frac{1}{2}}u\|^{\frac{1}{8}}\,\|\nabla
    u\|^{\frac{3(\alpha-2)}{4}+\frac{2}{3}}\,
    \|u\|_{3(\alpha-2)}^{\frac{\alpha-2}{2}}\,\|\Lambda^{\frac{5}{2}}u\|^{\frac{4}{3}+\frac{3}{8}}
    \\
    &\le C\, \|\Lambda^{\frac{1}{2}}u\|^{\frac{4-\alpha}{2}}\|\nabla
    u\|^{*} \|\Lambda^{\frac{5}{2}}u\|^{\frac{4}{3}+\frac{3}{8}}
    \\
    &\leq C\, \|\Lambda^\frac12 (u,v,\theta)\|^{*}\, \|\nabla
    (u,v,\theta)\|^{*}+ \varepsilon \|\Lambda^\frac52
    (u,v,\theta)\|^2,
  \end{aligned}
\end{equation*}
where we used again Gagliardo--Nirenberg's inequality, to get
\begin{gather*}
  \|u\|_\infty \leq C \|\nabla
  u\|^{\frac{2}{3}}\|\Lambda^{\frac{5}{2}}u\|^{\frac{1}{3}}\,\,\,
  \textrm{ and }\,\,\, \|\nabla u\|_6 \leq C
  \|\Lambda^{\frac{1}{2}}u\|^{\frac{1}{4}}\|\Lambda^{\frac{5}{2}}
  u\|^{\frac{3}{4}} 
  \\
  \textrm{ and }\,\, \|u\|_{3(\alpha-2)}\leq
  \|u\|^{\frac{4-\alpha}{2(\alpha-2)}}\|\nabla
  u\|^{\frac{3\alpha-8}{2(\alpha-2)}} \,\,\,\,\,\,\, \mbox{ (since
    $\alpha>8/3$)}.
\end{gather*}
Similarly, we get
\begin{multline*}
  \left| \int_{\R^3}\Lambda^{\frac{3}{2}}(|v|^{\beta-1}v)\cdot
    \Lambda^{\frac{3}{2}}  \dx \right| \leq
    \\
  \left\{
    \begin{aligned}
      &C\,\|\nabla v\|^{*} \,\|\Lambda^\frac12 v\|^{*}+ C\, \big(1+
      \|\nabla v\|^*\big)+ \varepsilon \||v|^\frac{\beta-1}{2}\,
      \nabla v\|^2+ \varepsilon \|\Lambda^\frac52 v\|^2, \,\,\,\,\,
      &&\mbox{ if $5/2\leq\beta \le 3$},
      \\
      &C\, \|\Lambda^\frac12 (u,v,\theta)\|^{*}\, \|\nabla
      (u,v,\theta)\|^{*}+ \varepsilon \|\Lambda^\frac52
      (u,v,\theta)\|^2, \,\,\,\,\, &&\mbox{ if $3< \beta <4$}.
    \end{aligned}
  \right.
\end{multline*}

As a consequence, we get
\begin{equation*} \label{e:H32}
  \begin{aligned}
    \frac{d}{dt} \|\Lambda^{\frac{3}{2}} (u, v, \theta) (t)
    \|^2 + & \big( \min\{\nu, \eta, \mu\} -
    \varepsilon\big)\|\Lambda^{\frac{5}{2}} (u, v, \theta)\|^2
    \\
    &\leq  C\,\|\Lambda^\frac12 (u, v, \theta)\|^*\, \|\nabla
    (u,v,\theta)\|^*
    {+ \varepsilon \||u|^\frac{\alpha-1}{2}\nabla
      u\|^2+\varepsilon\||v|^\frac{\beta-1}{2}\nabla v\|^2,}
  \end{aligned}
\end{equation*}
for every $0<t\leq T$, and $T<T^\ast$.
  
Adding \eqref{e:H32} with \eqref{e:H1}, we obtain in
particular that
\begin{equation*}
  \frac{d}{dt} \big(\|\nabla  (u,v,\theta)(t)\|^2 +\|\Lambda^{\frac32} (u,v,\theta)(t)\|^2\big) \le C_0,
\end{equation*}
from which we get
\begin{equation} \label{e:final}
  \sup_{0\leq t< T^*} \|\Lambda^{\frac32} (u,v,\theta)(t)\|^2 \le C_0 + T^*C_0 < + \infty.
\end{equation}
  
We then proved Theorem \ref{main}, with $c_0$ as in Corollary \ref{cor:monotone}, from \eqref{e:final}, the continuous embedding
    $\dot H^{\frac32}(\mathbb{R}^3)\hookrightarrow \mathrm{BMO}(\mathbb{R}^3)$,
Theorem~\ref{th:bu criterion} and Remark~\ref{rmk-Linfty-H3/2}. 

\medskip

\noindent\textbf{Acknowledgements.} The authors are
  members of the Gruppo Nazionale per l’Analisi Matematica, la
  Probabilità e le loro Applicazioni (GNAMPA) of the Istituto
  Nazionale di Alta Matematica (INdAM). Diego Berti was supported by
  PRIN grant 2020XB3EFL.  
  
  
\appendix
\section{No-damping case} \label{appendix}

This section contains an alternative proof, that relies on the
arguments of the previous sections, of a result already obtained by
\cite{Wang-Zhang-Pan}, that concerns the case of no dampings. 

{If we set $\sigma_1=\sigma_2=0$ in \eqref{TCM-gen}, we recover the
3D Tropical Climate Model without dampings.} Proceeding exactly as in the proof
of Theorem~\ref{th:bu criterion}, we can provide a blow-up criterion in
terms of integrability of powers of \textrm{BMO}-norms. Since in
\eqref{Hs-stima-diff-iniziale} without dampings the integrals $I_7$
and $I_8$ vanish, while for $I_1$--to--$I_6$, \eqref{I1}, \eqref{e:I4-I6}
and \eqref{e:I5} still apply, then, in this case, in place of
\eqref{e:I-dampings}, we have more directly that
\begin{equation}
\label{e:I-no-dampings}
  \begin{aligned} \frac{d}{dt} \big(& 1 + \|\Delta (u, v, \theta) (t)
\|^2\big) + \big(\min(\nu, \eta, \mu) - \varepsilon\big)\|\nabla
\Delta( u, v, \theta)\|^2 \\ &\leq C \Big( 1 +\|u\|_{{}_{\textrm{BMO}}}^{2}
+\|v\|_{{}_{\textrm{BMO}}}^{6} + \|\theta\|_{{}_{\textrm{BMO}}}^2 \Big) \big( 1 +
\|\Delta (u, v, \theta) \|^2\big),
  \end{aligned}
\end{equation}
and therefore the blow-up criterion is given by means of
\begin{equation}
  \tilde\Pi(T) \doteq \int_{0}^{T}\big(1
+\|u(t)\|_{{}_{\textrm{BMO}}}^{2} +\|v(t)\|_{{}_{\textrm{BMO}}}^{6} +
\|\theta(t)\|_{{}_{\textrm{BMO}}}^2\big)\,dt.
\end{equation}

In particular, we have the following blow-up criterion.

\begin{thm}[Blow-up criterion in the no-damping case]  \label{lem:bu-criterion-no-damp}
  Assume $(u_0,v_0, \theta_0)\in
H^2(\mathbb R^3)^3$, with
$\dv u_0=0$. Let $(u,v,\theta)(t)$ be the local strong solution of
\eqref{TCM-gen} with $\sigma_1=\sigma_2=0$. For $0< T^\ast < \infty$, we have that
\begin{gather*}
  \textrm{ $\|\Delta (u, v, \theta) (T)\| < +\infty$, for every $0<T<T^\ast$
    and $\limsup_{T \to {T^\ast}} \|\Delta (u, v, \theta) (T)\|=+\infty$},
\\
\intertext{if and only if}
  \textrm{$\tilde \Pi(T) < +\infty\,\,$, for every $\,\,0<T<T^\ast\,\,$ and $\,\,\tilde \Pi(T^\ast) =+\infty$.}
\end{gather*}
\end{thm}

For the $\dot H^{1}$-estimates, all the computations
coming from \eqref{stima-u-v-H1} apply, except for the fact that on
the left-hand side of it, no terms coming from the dampings
appear. Hence, inequality \eqref{stima-u-v-H1-chiusa} still holds true.
Concerning the $\dot H^{\frac12}$-estimates, from
\eqref{H1/2-stima-diff}, $K_4=K_5=0$, while all the other integrals
are the same. So the same estimates apply and hence, in place of
\eqref{H1/2-stima-diff.ii}, now we have
\begin{equation*}
    \frac{1}{2} \frac{d}{dt}  \|\Lambda^{\frac{1}{2}} (u, v, \theta) (t) \|^2  + \min\{\nu,
\eta, \mu\} \|\Lambda^{\frac{3}{2}} (u, v, \theta)\|^2 \leq
C\|\Lambda^{\frac{1}{2}} (u, v, \theta)\| \|\Lambda^{\frac{3}{2}} (u, v, \theta)\|^2.
\end{equation*}

This implies that we have the following companion of Lemma \ref{lem:monotone}:

\begin{lemma}[Monotonicity of $\dot H^1$- and $\dot H^{\frac12}$-norms in
  the no-damping case]
  Assume that $(u_0, v_0, \theta_0)\in H^2(\mathbb{R}^3)^3$,
with $\dv u_0=0$. Let $(u,v,\theta)(t)$ be the local solution of
\eqref{TCM-gen} with $\sigma_1=\sigma_2=0$.  If
   \begin{equation} \label{e:smallness-no-damp-one} \ell - C
\|\Lambda^{\frac{1}{2}} (u_0, v_0, \theta_0) \| >0,
    \end{equation} and
    \begin{equation} \label{e:smallness-no-damp-two} 2\, \ell -
C\big(\varepsilon + \|\Lambda^{\frac{1}{2}}
(u_0,v_0,\theta_0)\|^2\big) >0,
  \end{equation} then the following relations hold true
  \begin{equation*}
    \sup_{0\leq t <T^\ast} \|(u,v,\theta)(t)\|_{\dot H^{\frac12}} \le \|(u_0,v_0,\theta_0)\|_{\dot H^{\frac12}}, \,\, \textrm{ and }\,\,
\sup_{0\leq t <T^\ast} \|(u,v,\theta)(t)\|_{\dot H^1} \le \|(u_0,v_0,\theta_0)\|_{\dot H^1}.
\end{equation*}
\end{lemma}

Then, arguing as for the $\dot H^{\frac32}$-estimates in the case with damping,
from \eqref{e:H32} we get directly
\begin{equation*} \label{e:H32-nd}
    \frac{d}{dt} \|\Lambda^{\frac{3}{2}} (u, v, \theta) (t)
    \|^2 +  \big( \min\{\nu, \eta, \mu\} -
    \varepsilon\big)\|\Lambda^{\frac{5}{2}} (u, v, \theta)\|^2
    \leq  C\,\|\Lambda^\frac12 (u, v, \theta)\|^*\, \|\nabla
    (u,v,\theta)\|^*,\,\,\, \textrm{for $t\in [0, T^\ast)$.} 
  \end{equation*}
Suppose that $\|(u_0,v_0, \theta_0)\|_{\dot H^{\frac12}}$ is small enough
to fulfill \eqref{e:smallness-no-damp-two} and
\eqref{e:smallness-no-damp-one}.  Let $T>0$ be any real number less
than $T^\star$. Since $v\in L^\infty(0,T; \dot H^{\frac32}(\mathbb R^3)) \cap
L^2(0,T; \dot H^{\frac52}(\mathbb R^3))$, we deduce that $v\in
L^6(0,T; \dot H^{\frac32}(\mathbb R^3))$.  Using $\dot H^{\frac32}(\mathbb R^3)
\hookrightarrow \textrm{BMO}(\mathbb R^3)$, we conclude that $v\in L^6(0,T;
\textrm{BMO}(\mathbb R^3))$ and, in particular, $v\in L^2(0,T; \textrm{BMO}(\mathbb
R^3))$.  Similarly, $u, \theta\in L^2(0,T; \textrm{BMO}(\mathbb R^3))$.
Consequently, $\tilde \Pi(T^\ast)$ must be finite.
Hence, the blow-up criterion given by Theorem~\ref{lem:bu-criterion-no-damp}
works, and this completes the proof of Theorem~\ref{main}, in the no-damping case.

\end{document}